\documentclass[12pt,reqno]{amsart}

\usepackage{graphicx,graphics}
\usepackage{amsmath, amssymb, amscd, amsthm, comment, amsxtra}
\usepackage{amsfonts}
\usepackage{latexsym}
\usepackage[abs]{overpic}
\usepackage[pdftex, linktocpage=true]{hyperref}

\renewcommand{\marginpar}[1]{}
\catcode`\@=12

\def\Empty{}
\newcommand\oplabel[1]{
  \def\OpArg{#1} \ifx \OpArg\Empty {} \else
  	\label{#1}
  \fi}
		
\long\def\realfig#1#2#3#4{
\begin{figure}[htbp]
\centerline{\includegraphics[width=#4]{#2}}
\caption[#1]{#3}
\oplabel{#1}
\end{figure}}

\newcommand{\comm}[1]{}

\input{psfig}
\newtheorem{thm}{Theorem}[section]
\newtheorem{cor}[thm]{Corollary}
\newtheorem{lem}[thm]{Lemma}
\newtheorem{prop}[thm]{Proposition}

\newenvironment{pf}{\proof[\proofname]}{\endproof}
\newenvironment{pf*}[1]{\proof[#1]}{\endproof}
\usepackage{euscript}

\usepackage[OT2,OT1]{fontenc}

\newcommand{\cal}[1]{{\mathcal #1}}

\newcommand{\ignore}[1]{{}}

\theoremstyle{definition}
\newtheorem{defn}{Definition}[section]

\theoremstyle{remark}

\newcommand{\diam}{\operatorname{diam}}

\newcommand{\dist}{\operatorname{dist}}

\renewcommand{\mod}{\operatorname{mod}}
\newcommand{\tl}{\tilde}
\newcommand{\h}{\hat}
\newcommand{\wtl}{\widetilde}
\newcommand{\eps}{\epsilon}

\newcommand{\vareps}{\varepsilon}

\newcommand{\aaa}[1]{{{\mathbf{#1}}}}

\newcommand{\hol}{{\aaa H}}

\numberwithin{equation}{section}
\newcommand{\thmref}[1]{Theorem~\ref{#1}}
\newcommand{\propref}[1]{Proposition~\ref{#1}}

\newcommand{\lemref}[1]{Lemma~\ref{#1}}

\newcommand{\cRG}{{\EuScript R}}

\newcommand{\cQ}{{\cal Q}}

\newcommand{\cA}{{\cal A}}
\newcommand{\cU}{{\cal U}}
\newcommand{\cW}{{\cal W}}

\newcommand{\cV}{{\cal V}}

\newcommand{\cI}{{\cal I}}

\newcommand{\cP}{{\cal P}}
\newcommand{\cC}{{\cal C}}

\newcommand{\cR}{{\cal R}}
\newcommand{\cL}{{\cal L}}
\newcommand{\cD}{{\cal D}}

\newcommand{\cS}{{\cal S}}

\newcommand{\CC}{{\Bbb C}}
\newcommand{\RR}{{\Bbb R}}
\newcommand{\TT}{{\Bbb T}}
\newcommand{\ZZ}{{\Bbb Z}}
\newcommand{\NN}{{\Bbb N}}
\newcommand{\DD}{{\Bbb D}}

\renewcommand{\AA}{{\Bbb A}}
\newcommand{\QQ}{{\Bbb Q}}

\newcommand{\bH}{{\mathbf H}}
\newcommand{\bT}{{\mathbf T}}
\newcommand{\bF}{{\mathbf F}}
\newcommand{\bU}{{\mathbf U}}
\newcommand{\bB}{{\mathbf B}}
\newcommand{\bC}{{\mathbf C}}

\newcommand{\bD}{{\mathbf D}}

\newcommand{\cren}{\cR_{\text cyl}}

\begin{document}
\addtolength{\evensidemargin}{-0.7in}
\addtolength{\oddsidemargin}{-0.7in}

\title[Almost commuting pairs]{Renormalization of almost commuting pairs}

\author{D. Gaidashev and M. Yampolsky }
\date{\today}
\maketitle

\begin{abstract}
In this paper we give a new proof of hyperbolicity of renormalization of critical circle maps using the formalism of almost-commuting pairs. We extend renormalization to two-dimensional dissipative maps of the annulus which are small perturbations of one-dimensional critical circle maps. Finally, we demonstrate that a two-dimensional map which lies in the stable set of the renormalization operator possesses attractor which is topologically a circle. Such a circle is {\it critical}: the dynamics on it is topologically, but not smoothly, conjugate to a rigid rotation.
\end{abstract}

\section{Preliminaries}
\subsection{Introduction}

Our motivation in this paper comes from  the study of attractors of small two-dimensional perturbations of critical circle 
maps. Let us recall, that
a critical circle map $f$ is a $C^3$-smooth orientation preserving homeomorphism of the circle $\TT\equiv\RR/\ZZ$ which
has a single critical point $x_0\in\TT$ whose order  $n$ is an odd integer. 
To fix the ideas, we will set $x_0=0$,
and will assume that $n=3$. By way of example, consider the
two-parameter {\it Arnold's family}
$$f_{a,\omega}(x)=x-\frac{a}{2\pi} \sin 2\pi x+\omega.$$
Note that each $f_{a,\omega}$ commutes with the unit translation,
$$f_{a,\omega}(x+1)=f_{a,\omega}(x)+1,$$
and hence it projects to a well-defined map of the circle $\TT\equiv \RR/\ZZ$, which we denote 
$\hat f_{a,\omega}$. For $|a|< 1$, this map is  an analytic diffeomorphism, and for $|a|=1$ it is a critical circle map.
This  illustrates  the fact that a generic analytic homeomorphism of the circle which lies on the
boundary of the set of analytic diffeomorphisms is a critical circle map.

For a circle homeomorphism $f$, we will denote $\rho(f)\in\TT$ its rotation number. 
For a lift $\bar f:\RR\to\RR$, we obtain a representative of $\rho(f)$ given by $\lim\bar f^n(x)/n$. We denote it by 
$\rho(\bar f)\in\RR$.
As was shown by Yoccoz in \cite{Yoc}, every critical circle map $f$ with $\rho(f)\notin\QQ$ is topologically conjugate to the
rigid rotation
$$R_{\rho(f)}(x)\equiv x+\rho(f)\mod \ZZ.$$
Identifying $\rho(f)$ with its 
representative in $[0,1)$, we can represent $\rho(f)$ as a contined fraction with positive terms
\begin{equation}
\label{rotation-number}
\rho(f)=\cfrac{1}{r_0+\cfrac{1}{r_1+\cfrac{1}{r_2+\dotsb}}}
\end{equation}
For convenience, further on we will abbreviate this expression as $[r_0,r_1,r_2,\ldots]$.
The numbers $r_i$ are determined uniquely if and only if $\rho(f)$ is
irrational. In this case we shall say that $\rho(f)$ (or $f$ itself) is of the type bounded by $B$
if $\sup r_i\leq B$; it is of a periodic type if the sequence $\{r_i\}$ is periodic.

Let $\AA_r$ denote the annulus $\{(x,y)\in\RR^2,\;|y|<r\}/\ZZ\supset \TT$ and let $F$ be a real-analytic map 
$$F:\AA_r\to\AA_r.$$
We let $$\Lambda(F)=\cap_{n\in\NN}F^n(\AA_r),$$ 
and refer to it as the {\it attractor} of $F$; we further call it a minimal attractor when the restriction $F|_{\Lambda(F)}$ is minimal.
In the case when $f$ is a map of the circle, we can trivially extend it to the second coordinate, setting
$F_f(x,y)=(f(x),0)$; in this case, $\Lambda(F)=\TT$. Suppose, $f$ is an analytic diffeomorphism of $\TT$.
Considerations of normal hyperbolicity imply that if $G$ is a sufficiently small smooth perturbations of $F_f$,
 the attractor $\Lambda(G)$ is a smooth circle, and furthermore, when $\Lambda(G)$ is minimal, the dynamics of $G$ on $\Lambda(G)$ is 
smoothly conjugate to the irrational rotation. 
Recently, E.~Pujals \cite{Puj} asked a question, whether, when considering small perturbations of critical circle maps, one would observe
``critical'' invariant circles: that is, topological circles $\Lambda(G)$ on which the dynamics is topologically, but not smoothly,
conjugate to an irrational rotation. This question can be asked in a typical low-parameter family of perturbations of critical circle maps, or for a specific family of examples. Pujals proposed looking at the perturbed Arnold family, consisting of quotients under 
$x\equiv x+1$ of maps of the form
 $$(f_{a,\omega}(x)+y,\eps(f_{a,\omega}(x)-x+y)),$$
where $\eps$ is a small parameter. Here, if we, for instance, fix the rotation number $$\rho_*=(\sqrt{5}-1)/2=[1,1,1,1,\ldots],$$
one would expect that possessing a critical circle with rotation number $\rho_*$ would be a codimension $1$ phenomenon, occuring on the boundary of the set in which the attractor is a non-critical circle with the same rotation number.

In this paper, we confirm that critical circles exist in typical families, and explain the criticality phenomenon in terms of 
hyperbolicity of renormalization, which is a subject of this paper in it own right. Briefly, maps of the annulus with a critical circle
with rotation number $\rho_*$ (for example) lie in the stable manifold of the one-dimensional hyperbolic fixed point of renormalization.

Of course, renormalization of critical circle maps is a classical subject,
and one of the central themes in the development of modern one-dimensional dynamics. We refer the reader to the papers \cite{Ya3,Ya4}
of the second author in which the main renormalization conjectures, known as Lanford's Program, were proved. The preceding historical
development of the subject is described in \cite{Ya3}. 
The ``classical'' definition of renormalization of critical circle maps uses the language of {\it commuting pairs}, as described below.
Analytic commuting pairs provided the setting for proving the existence of renormalization horseshoe attractor \cite{dF2,dFdM2,Ya4}. 
However, there was a conceptual difficulty in proving hyperbolicity in this setting, as the space of analytic commuting pairs does
 not possess a natural structure  of a Banach manifold.

This difficulty was finessed by the
second author using a concept of {\it cylinder renormalization}, introduced in \cite{Ya3}. Cylinder renormalization operator $\cren$
has two key properties, necessary for the study of hyperbolic properties of the renormalization horseshoe attractor:
\begin{enumerate}
\item $\cren$ acts on a Banach manifold (of analytic maps of the circle, whose domain of analyticity includes a certain fixed annulus);
\item the operator $\cren$ is smooth (in fact, analytic). 
\end{enumerate}
Cylinder renormalization has since become an important tool in one-dimensional renormalization theory. It applies to analytic maps with
Siegel  disks \cite{Ya5,GaY}; in the limiting case it becomes the all-important {\it parabolic renormalization} \cite{EY,IS}; and 
very recently it has been applied to the study of critical circle maps with non-integer critical exponents \cite{GoY}. 

However, the question of proving hyperbolicity in the setting of commuting pairs has remained relevant. One of the main reasons for this is that cylinder renormalization does not extend readily to small two-dimensional perturbations of critical circle maps.
The definition of $\cren$ relies on the Uniformization Theorem of doubly-connected domains of one-dimensional Complex Analysis. This 
definition does not naturally generalize to two-dimensional maps. In this paper, we revisit the problem of hyperbolicity of
 renormalization. As will be seen in the next section, we use a ``classical'' definition of renormalization and
the definition of a Banach manifold in which renormalization becomes smooth (analytic) -- and thus satisfy the above
conditions (1)-(2) for commuting pairs.

We then give a new proof of renormalization hyperbolicity -- in the ``classical'' setting of commuting pairs. This allows us to 
apply our renormalization to small two-dimensional perturbations
of critical circle maps. We find a suitable smooth extension of renormalization to dissipative maps of the annulus in two dimension,
and prove renormalization hyperbolicity for such maps.
Finally, we apply our renormalization results to the study of dissipative attractors of small perturbations of critical circle maps,
to prove a version of Pujals' conjectures.

\subsection{Commuting pairs}
 As discussed in some detail in \cite{Ya3}, the space of critical circle maps is
ill-suited to define renormalization. The pioneering works on the subject (\cite{ORSS} and \cite{FKS})
 circumvented this difficulty
by replacing critical circle maps with different objects:
\begin{defn}
A  {\it  $C^r$-smooth (or $C^\omega$) critical commuting pair} $\zeta=(\eta,\xi)$ consists of two 
$C^r$-smooth  (or $C^\omega$) orientation preserving interval homeomorphisms 
$\eta:I_\eta\to \eta(I_\eta),\;
\xi:I_{\xi}\to \xi(I_\xi)$, where
\begin{itemize}
\item[(I)]{$I_\eta=[0,\xi(0)],\; I_\xi=[\eta(0),0]$;}
\item[(II)]{Both $\eta$ and $\xi$ have homeomorphic extensions to interval
neighborhoods of their respective domains {\it with the same degree of
smoothness}, that is $C^r$ (or $C^\omega$), which commute, 
$\eta\circ\xi=\xi\circ\eta$;}
\item[(III)]{$\xi\circ\eta(0)\in I_\eta$;}
\item[(IV)]{$\eta'(x)\ne 0\ne \xi'(y) $, for all $x\in I_\eta\setminus\{0\}$,
 and all $y\in I_\xi\setminus\{0\}$;}
\item[(V)] each of the maps $\eta$ and $\xi$ has a cubic critical point at $0$:
$$\eta'(0)=\eta''(0)=\xi'(0)=\xi''(0)=0,\text{ and }\eta'''(0)\neq 0\neq \xi'''(0) .$$
\end{itemize}
\end{defn}

\realfig{compair}{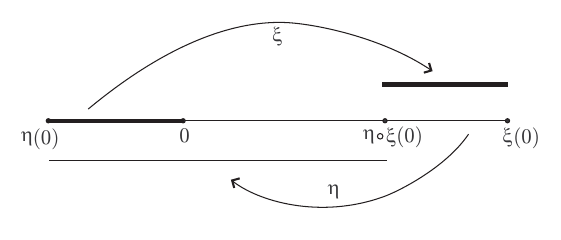}{A commuting pair}{0.8\textwidth}

\noindent
The commutation condition allows one to ``seamlessly'' iterate the extensions of the maps of a commuting pair. 

Given a critical commuting pair $\zeta=(\eta,\xi)$
we can regard the interval $I=[\eta(0),\xi\circ\eta(0)]$ as a circle, identifying 
$\eta(0)$ and $\xi\circ\eta(0)$ and define $f_\zeta:I\to I$ by 
$$f_\zeta=\left\{\begin{array}{l}
                    \eta\circ\xi(x) \text{ for }x\in [\eta(0),0]\\
                    \eta(x)\text{ for } x\in [0,\xi\circ\eta(0)]
                \end{array}\right.
$$
The mapping $\xi$ extends to a $C^r$- (or $C^\omega$-) diffeomorphism of open
neighborhoods of  $\eta(0)$ and $\xi\circ\eta(0)$. Using it as a local chart we turn 
the interval $I$ into a closed one-dimensional manifold $M$. Condition (II) above implies that
the mapping $f_\zeta$ projects to a well-defined $C^3$-smooth homeomorphism $F_\zeta:M\to M$.
Identifying $M$ with the circle by a diffeomorphism $\phi:M\to \TT$
we recover a critical circle mapping $f^\phi=\phi\circ F_\zeta\circ\phi^{-1}$.
The critical circle mappings corresponding to two different choices
of $\phi$ are conjugated by a diffeomorphism, and thus we recovered
a $C^r$- (or $C^\omega$) smooth conjugacy class of circle mappings from a critical commuting
pair.

Let $f$ be a critical circle mapping, whose rotation number $\rho$
has a continued fraction expansion (\ref{rotation-number}) with
at least $m+1$ terms, and let $p_m/q_m=[r_0,\ldots,r_{m-1}]$. Let $I_m$ denote the closed arc of the circle connecting $0$ with $f^{q_m}(0)$ and not containing $f^{q_{m+1}}(0)$.  The pair of iterates $f^{q_{m+1}}$
and $f^{q_m}$ restricted to the circle arcs $I_m$ and $I_{m+1}$
correspondingly can be viewed as a critical commuting pair  in
the following way.
Let $\bar f$ be the lift of $f$ to the real line satisfying $\bar f '(0)=0$,
and $0<\bar f (0)<1$. For each $m>0$ let $\bar I_m\subset \Bbb R$ 
denote the closed 
interval adjacent to zero which projects down to the interval $I_m$.
Let $\tau :\Bbb R\to \Bbb R$ denote the translation $x\mapsto x+1$.
Let $\eta :\bar I_m\to \Bbb R$, $\xi:\bar I_{m+1}\to \Bbb R$ be given by
$\eta\equiv \tau^{-p_{m+1}}\circ\bar f^{q_{m+1}}$,
$\xi\equiv \tau^{-p_m}\circ\bar f^{q_m}$. Then the pair of maps
$(\eta|_{\bar I_m},\xi|_{\bar I_{m+1}})$ forms a critical commuting pair
corresponding to $(f^{q_{m+1}}|_{I_m},f^{q_m}|_{I_{m+1}})$.
Henceforth we shall  simply denote this commuting pair by
\begin{equation}
\label{real1}
(f^{q_{m+1}}|_{I_m},f^{q_m}|_{I_{m+1}}).
\end{equation}

The {\it height} $\chi(\zeta)$
of a critical commuting pair $\zeta=(\eta,\xi)$ is equal to $r$,
if 
$$0\in [\eta^r(\xi(0)),\eta^{r+1}(\xi(0))].$$
 If no such $r$ exists,
we set $\chi(\zeta)=\infty$, in this case the map $\eta|_{I_\eta}$ has a 
fixed point.  For a pair $\zeta$ with $\chi(\zeta)=r<\infty$ one verifies directly that the
mappings $\eta|_{[0,\eta^r(\xi(0))]}$ and $\eta^r\circ\xi|_{I_\xi}$
again form a commuting pair.
For a commuting pair $\zeta=(\eta,\xi)$ we will denote by 
$\wtl\zeta$ the pair $(\wtl\eta|_{\wtl{I_\eta}},\wtl\xi|_{\wtl{I_\xi}})$
where tilde  means rescaling by the linear factor $\xi(0)$;
$$\wtl\zeta(z)=((\xi(0))^{-1}\eta(\xi(0) z),(\xi(0))^{-1}\xi(\xi(0) z)).$$
Note that the domain of definition of the first map of the rescaled pair is $\wtl{I_\eta}=[0,1]$.

\begin{defn} We say that a real commuting pair $\zeta=(\eta,\xi)$ is {\it renormalizable} if $\chi(\zeta)<\infty$.
The {\it renormalization} of a renormalizable commuting pair $\zeta=(\eta,
\xi)$ is the commuting pair
\begin{center}
${\cal{R}}\zeta=(
\widetilde{\eta^r\circ\xi}|_{ \widetilde{I_{\xi}}},\; \widetilde\eta|_{\widetilde{[0,\eta^r(\xi(0))]}}).$
\end{center}
\end{defn}

\noindent
The non-rescaled pair $(\eta^r\circ\xi|_{I_\xi},\eta|_{[0,\eta^r(\xi(0))]})$ will be referred to as the 
{\it pre-renormalization} $p{\cal R}\zeta$ of the commuting pair $\zeta=(\eta,\xi)$. Suppose $\{\zeta_i\}_{i=1}^{k-1}$ is a sequence of renormalizable pairs such that $\zeta_0=\zeta$ and $\zeta_i=p\cR\zeta_{i-1}$. We call $\zeta_k=p\cR\zeta_{k-1}$ the $k$-th pre-renormalization of $\zeta$; and $\widetilde{\zeta_k}$ the $k$-th  renormalization of $\zeta$ and write
$$\zeta_k=p\cR^k\zeta,\;\widetilde{\zeta_k}=\cR^k\zeta.$$
Let $\zeta_k=(\eta_k,\xi_k)$. The domains of $\eta_k$ and $\xi_k$ will be denoted $I_k$ and $J_k$ correspondingly.

For a pair $\zeta$ we define its {\it rotation number} $\rho(\zeta)\in[0,1]$ to be equal to the 
continued fraction $[r_0,r_1,\ldots]$ where $r_i=\chi({\cal R}^i\zeta)$. 
In this definition $1/\infty$ is understood as $0$, hence a rotation number is rational
if and only if only finitely many renormalizations of $\zeta$ are defined;
if $\chi(\zeta)=\infty$, $\rho(\zeta)=0$.
Thus defined, the rotation number of a commuting pair can be viewed as a rotation number in
the usual sense:
\begin{prop}
\label{rotation number}
The rotation number of the mapping $F_\zeta$ is equal to $\rho(\zeta)$.
\end{prop}

\noindent
There is an  advantage in defining $\rho(\zeta)$ using a sequence of heights in
removing the ambiguity in prescribing a continued fraction expansion to rational rotation numbers
in a renormalization-natural way.

\subsection{Dynamical partitions and real {\it a priori} bounds}
\label{sec:partition}
We need to recall the definition of a dynamical partition, which becomes somewhat technical in the language of commuting pairs.
Consider the space $\cI$ of multi-indices $\bar s=(a_1,b_1,a_2,b_2,\ldots,a_m,b_m)$ where $a_j\in \NN$ for $2\leq m$, $a_1\in\NN\cup\{0\}$,
$b_j\in\NN$ for $1\leq j\leq m-1$, and $b_m\in\NN\cup\{0\}$. 
We introduce a partial ordering on multi-indices:
$\bar s\succ \bar t$ if $\bar s=(a_1,b_1,a_2,b_2,\ldots,a_m,b_m)$, $\bar t=(a_1,b_1,\ldots,a_k,b_k,c,d)$, where $k<m$ and 
either $c< a_{k+1}$ and $ d=0$ or $c=a_{k+1}$ and $d< b_{k+1}$.

For a pair of maps $\zeta=(\eta,\xi)$ and $\bar s $ as above we will denote 
$$\zeta^{\bar s}\equiv\xi^{b_m}\circ\eta^{a_m}\circ\cdots\circ\xi^{b_2}\circ\eta^{a_2}\circ\xi^{b_1}\circ \eta^{a_1}.$$
Similarly, 
$$\zeta^{-\bar s}\equiv (\zeta^{\bar s})^{-1}=(\eta^{a_1})^{-1}\circ(\xi^{b_1})^{-1}\circ\cdots\circ(\eta^{a_m})^{-1}\circ (\xi^{b_m})^{-1}.$$

Let us define the {\it $n$-th dynamical partition} $\cP_n$ of $\zeta=(\eta,\xi)$ which is at least $n$ times renormalizable.
Namely, consider the $n$-th pre-renormalization 
$$\zeta_{n}=(\eta_{n}|_{I_n},\xi_{n}|_{J_n}),\text{ where }I_n=[0,\xi_{n}(0)]\text{ and }J_n=[0,\eta_{n}(0)].$$
We let $\bar s_n,\bar t_n\in\cI$ be defined by
\begin{equation}
  \label{sntn}
  \eta_{n}=\zeta^{\bar s_n},\text{ and }\xi_{n}=\zeta^{\bar t_n}.
  \end{equation}
Now consider the collection of intervals
$$\cP_n\equiv \{\zeta^{\bar w}(I_n)\text{ for all }\bar w\prec \bar s_n\text{ and }\zeta^{\bar w}(J_n)\text{ for all }\bar w\prec \bar t_n\}.$$
It is easy to see that:
\begin{itemize}
\item[(a)] $\underset{H\in\cP_n}\cup H=[\eta(0),\xi(0)]$;
\item[(b)] for any two distinct elements $H_1$ and $H_2$ of $\cP_n$, the interiors of $H_1$ and $H_2$ are disjoint.
\end{itemize}
\realfig{partition}{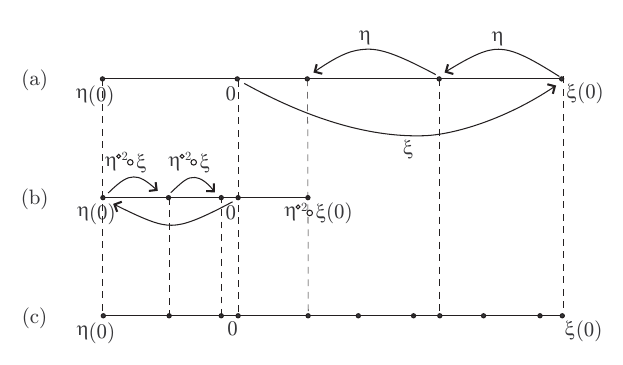}{The 1-st and 2-nd dynamical partitions for a pair $\zeta$ with $\rho(\zeta)=[2,2,\ldots]$:
(a) forming the partition of level 1; (b) the 1-st dynamical partition of the pre-renormalization $p\cR\zeta$; (c) the 2-nd
dynamical partition of $\zeta$.
}{0.8\textwidth}
We denote $\overline{\cP_n}$ the set of boundary points of the $n$-th dynamical partition\footnote{We note that typically, in the literature, dynamical partitions are described for critical circle maps, rather than commuting pairs, as this leads to a simpler notation. The translation from one notation to the other is straightforward, see e.g.  \cite{dF1,dF2}.}.

Successive renormalizations of a $C^3$-smooth commuting pair with an irrational rotation number form a pre-compact family,
all of the limit points of which are {\it analytic}. For a  strong version of this statement, known as {\it real a priori bounds},
 see \cite{dFdM1}; we will need the following consequence of compactness:
\begin{prop}
\label{realbounds}
There exists a universal constant $C_0>1$ such that the following holds. Let $S$ be a compact set of $C^3$-smooth
commuting pairs (note that $S$ could consist of a single pair). Then there exists $N=N(S)$ such that for all $n\geq N$ the following holds. Let $\zeta\in S$ be at least $n$ times renormalizable. Let $I$ and  $J$ be two {\it adjacent} intervals of the 
$n$-th dynamical partition of $\zeta$. Then $I$ and $J$ are $C_0$-commensurable:
$$\frac{1}{C_0}|I|<|J|<C_0|I|.$$
In particular,
denoting $p\cR^n\zeta=(\eta',\xi')$, we have
$$\frac{1}{C_0}|I_{\xi'}|<|I_{\eta'}|<C_0|I_{\xi'}|.$$ 
\end{prop}

\subsection{Renormalization horseshoe}
In \cite{Ya2} we constructed a horseshoe attractor for renormalization of analytic maps.
Denote $\bar\Sigma$ the space of bi-infinite sequences 
$$(\ldots,r_{-k},\ldots,r_{-1},r_0,r_1,\ldots,r_k,\ldots)\text{ with }r_i\in\NN\cup\{\infty\}$$
equipped with the weak topology. To talk about convergence of analytic commuting pairs $\zeta=(\eta,\xi)$ it will be convenient to consider
them as pairs of analytic maps with their domains of definition:
$$\eta:D_\eta\to \CC,\; \xi:D_\xi\to\CC,$$
where $D_\eta$ and $D_\xi$ are real-symmetric topological disks containing $I_\eta$ and $I_\xi$ respectively.
When convenient, we will write $\zeta=(\eta|_{D_\eta},\xi|_{D_\xi})$. We recall that a sequence $\zeta_n=(\eta_n|_{D_{\eta_n}},\xi_n|_{D_{\xi_n}})$ converges to
$\zeta_\infty=(\eta_\infty|_{D_{\eta_\infty}},\xi_\infty|_{D_{\xi_\infty}})$ in the sense of Carath{\'e}odory if:
\begin{itemize}
\item for each Hausdorff limit point $K$ of the sequence $\CC\setminus D_{\eta_n}$, the domain $D_{\eta_\infty}$ is a connected component of $\CC\setminus K$; and similarly for $D_{\xi_\infty}$;
 \item the maps $\eta_n\to\eta_\infty$ and $\xi_n\to\xi_\infty$ uniformly on compact subsets of $D_{\eta_\infty}$, $D_{\xi_\infty}$ respectively. 
\end{itemize}
We refer the reader to \cite{McM1} where Carath{\'e}odory convergence is introduced in a renormalization context, and \cite{Ya3} which contains a detailed discussion of Carath{\'e}odory convergence for the space of analytic commuting pairs, and, in particular, introduces the corresponding topology on this space.

\begin{thm}[{\bf Renormalization horseshoe \cite{Ya2}}]
\label{existence of attractor}
There exists an $\cal R$-invariant set ${\cal X}$ consisting of analytic commuting pairs with irrational 
rotation numbers with the following properties. The operator $\cR$ continuously extends to the closure (in the sense of Carath{\'e}odory convergence)
$$\cA\equiv\overline{\cal X}$$
and the action of $\cal R$ on $\cal A$ is topologically
conjugate to the two-sided shift $\sigma:\bar \Sigma\to\bar\Sigma$:
$$i\circ{\cal R}\circ i^{-1}=\sigma$$
so that if $\zeta=i^{-1}(\ldots,r_{-k},\ldots,r_{-1},r_0,r_1,\ldots,r_k,\ldots)$ then
$\rho(\zeta)=[r_0,r_1,\ldots,r_k,\ldots]$. For any analytic commuting pair $\zeta$ with an irrational rotation number
we have
$${\cal R}^n\zeta\to \cal A$$
in the Carath{\'e}odory sense. Moreover, for any two analytic commuting pairs $\zeta$, $\zeta'$ with $\rho(\zeta)=\rho(\zeta')$
we have 
$$\dist({\cal R}^n\zeta,{\cal R}^n\zeta')\to 0$$
for the uniform distance between analytic extensions of the renormalized pairs on compact sets.
\end{thm}
We will denote $\cA_B$ the subset of the attractor consisting of pairs with rotation numbers of a type bounded by $B$. Its existence,
and the corresponding version of \thmref{existence of attractor} was shown by E.~de~Faria (see \cite{dF1,dF2} and also \cite{dFdM2}).

Let $\zeta=(\eta,\xi)$ be a commuting pair such that $\xi(0)=1$. Denote $C^0([0,1])$ the Banach space of bounded $C^0$ functions on the 
interval $[0,1]$ with the uniform norm. We can identify $\zeta$ with a point in $\RR\times C^0([0,1])\times C^0([0,1])$ by 
\begin{equation}
\label{corresp1}
\zeta\mapsto (\eta(0),\eta(x),\frac{1}{\eta(0)}\xi(\eta(0)x)).
\end{equation}
This induces a distance on the set of commuting pairs, which we denote $\dist_{C^0}$.
We note that the following has been recently proven by W.~de~Melo and P.~Guarino \cite{GdM}:
\begin{thm}
\label{thmGuarino}
There exists $\delta>0$ such that the following holds. Let $\zeta_1$ and $\zeta_2$ be two $C^3$-smooth commuting pairs with the same 
 irrational rotation number $\rho=\rho(\zeta_1)=\rho(\zeta_2)$ of bounded type.  Then there exists $C>0$ such that
$$\dist_{C^0}(\cR^n\zeta_1,\cR^n\zeta_2)<C(1+\delta)^{-n}. $$
\end{thm}

\subsection{Spaces of analytic almost commuting pairs}
Because of the commutation condition, there is no natural Banach manifold structure on the space of analytic commuting pairs.
 
However, there is one on the space of $C^r$-smooth commuting pairs with $r\geq 3$, considered modulo an affine conjugacy.
Indeed, pick the unique representative $\zeta=(\eta,\xi)$ of an affine conjugacy class, which is 
given by the normalization $\xi(0)=1$. Let $C^r([0,1])$ denote the Banach space of $C^r$-smooth functions on $[0,1]$ with the norm
$$||f||_{C^r}=\sum_{k=0}^r\sup_{x\in[0,1]}\left|\frac{d^k}{dx^k}f\right|.$$
As above, identify $C^r$-smooth commuting pairs with a subset of  $\RR\times C^r([0,1])\times C^r([0,1])$ via (\ref{corresp1}). It is possible to 
show that this subset has a submanifold structure. 
Clearly, the space of $C^r$-smooth commuting pairs is renormalization-invariant. However, it is 
an elementary exercise to show that the operator $\cR$ is {\it not} differentiable in the space of $C^r$-smooth pairs (indeed, composition,
considered as an operator $C^r\times C^r\to C^r$ is not differentiable). Thus the setting of $C^r$-smooth commuting pairs is equally unsuitable
for the study of the hyperbolic properties of $\cR$.

We, therefore, take a different path.
The principal  object in our approach to critical circle maps is the following space:
\begin{defn}
\label{def1}
The space $\bB$ consists of $C^3$-smooth commuting pairs $\zeta=(\eta,\xi)$,
such that the maps $\eta$,$\xi$ 
are complex-analytic on some neighborhoods of their intervals of definiton.
We call the elements of $\bB$ {\it analytic almost commuting pairs} or simply {\it almost commuting pairs}.

\end{defn}
A version of this ``classical'' approach was first used in the computer-assisted proof of renormalization hyperbolicity 
by Mestel \cite{Mes}, although, it has not received any further development in the literature since.

We claim that an equivalent way of describing this space is the following:
\begin{defn}
\label{def2}
The space $\bB$ consists of pairs of non-decreasing interval maps
$$\eta:[0,\xi(0)]\to[\eta(0),\eta\circ\xi(0)],\;\xi:[\eta(0),0]\to[\xi\circ\eta(0),\xi(0)] $$ 
which have the following properties:
\begin{enumerate}
\item  there exists an open neighborhood of the interval $[0,\xi(0)]$ on which the map $\eta$ is analytic, with a single critical point
of order $3$ at the origin;
\item similarly, there exists an open neighborhood of the interval $[\eta(0),0]$ on which the map $\xi$ is analytic, with a single critical 
point of order $3$ at the origin;
\item the commutator
$$[\eta,\xi](x)\equiv \eta\circ\xi(x)-\xi\circ\eta(x)=o(x^3)\text{ at }x=0.$$
\end{enumerate}
\end{defn}

It is evident that a pair satisfying Definition \ref{def1} also satisfies Definition \ref{def2}. To prove the converse, 
let $(\eta,\xi)$ be a pair satisfying \ref{def2}. Consider the extension of $\eta$ to a function $\tl\eta$ defined in a
 neighborhood of $0$, which is given by $\eta$ on $[0,\xi(0)]$ and by $\xi^{-1}\circ\eta\circ\xi$ on $[\eta(0),0]$.
Since $\xi$ is a local diffeomorphism away from the origin, we have
$$\eta(x)-\tl\eta(x)\sim \xi\circ\eta(x)-\eta\circ\xi(x)=o(x^3).$$
Hence, $\tl\eta$ is a $C^3$-smooth extension of $\eta$ to a neighborhood of $[0,\xi(0)]$, which commutes with the analytic extension of $\xi$, and the claim is proved.

Suppose, $B$ is a complex Banach space whose elements are functions of a complex variable. Let us say that
the {\it real slice} of $B$ is the real Banach space $B^\RR$ consisting of the real-symmetric elements of $B$.
If $X$ is a Banach manifold modelled on $B$ with the atlas $\{\Psi_\gamma\}$
we shall say that $X$ is {\it real-symmetric} if $\Psi_{\gamma_1}\circ\Psi_{\gamma_2}^{-1}(B^\RR)\subset B^\RR$ for any pair of indices $\gamma_1$, $\gamma_2$. The {\it real slice of $X$} is then defined as the real
Banach manifold $X^\RR\subset X$ given by $\Psi_\gamma^{-1}(B^\RR)$ in a local chart $\Psi_\gamma$.
An operator $A$ defined on a subset of $X$ is {\it real-symmetric} if $A(X^\RR)\subset X^\RR$. 

\begin{defn}
For a choice of topological disks $D\supset [0,1]$, $E$,
we let $\bB^{D,E}_0$ consists of pairs in $\bB$ whose maps $\eta$ and $\xi$ have bounded analytic continuations
 to $D$ and $E$ correspondingly, such that 
$[\eta(0),0]\subset E$, and such that $0$ is the only critical point of $\eta$ and $\xi$ on a neighborhood of $I_\eta$, $I_\xi$ respectively. We view it as a subset of the real slice of the complex Banach space $C^\omega(D)\times C^\omega(E)$ 
where $C^\omega(W)$ denotes
the space of bounded holomorphic functions on $W$ with the uniform norm. 
Finally, denote $\bB^{D,E}$ the space of pairs in $\bB^{D,E}_0$  with  further normalization conditions $\xi(0)=1$, and 
$\frac{1}{2C_0}<|\eta(0)|<2C_0,$ where $C_0$ is as in \propref{realbounds}.
\end{defn}

\begin{prop}
\label{Banach1}
With these norms, the space $\bB^{D,E}$ is a real Banach manifold, modeled on  a finite-codimensional subspace of the real slice of the Banach space
$C^\omega(D)\times C^\omega(E)$.
\end{prop}
\begin{proof}Firstly, note that the conditions $\eta'(0)=\eta''(0)=\xi'(0)=\xi''(0)=0$ define a Banach subspace of $C^\omega(D)\times C^\omega(E)$. Furthermore, by the Argument Principle and considerations of continuity, the conditions
  \begin{itemize}
    \item $\frac{1}{2C_0}<|\eta(0)|<2C_0$,
\item   $\eta'''(0)\neq 0$, $\xi'''(0)\neq 0$,
\item   $\eta'(x)>0$ and $\xi'(x)>0$ on $I_\eta\setminus\{0\}$ and $I_\xi\setminus\{0\}$ respectively, and
  \item $\eta'(z)\neq 0$, $\xi'(z)\neq 0$ for $z\neq 0$ on neighborhoods of $I_\eta$, $I_\xi$  respectively,
\end{itemize}
    define an open subset $\cW$ of this Banach subspace. 

    The conditions $\eta'(0)=\eta''(0)=\xi'(0)=\xi''(0)=0$  imply that $$(\eta \circ \xi)^{(n)}(0)=(\xi \circ \eta)^{(n)}(0)=0\text{ for }n=1,2.$$
    Thus, the space $\bB^{D,E}$ is the preimage $\bF^{-1}(0)\subset \cW$, where the map $\bF:\cW\to\RR^3$ is given by
    $$\bF=(F_1,F_2,F_3)\equiv (\eta\circ\xi(0)-\xi\circ\eta(0),(\eta\circ\xi)'''(0)-(\xi\circ\eta)'''(0),\xi(0)-1).$$
Let $k\geq 3$, and write 
$$\eta(x)={\mathbf u}_k x^k + {\mathbf u}_{k+1} x^{k+1} + q(x),\; \xi(x)={\mathbf v}_0+s(x),\text{ where }{\mathbf v}_0 \equiv \xi(0),\; {\mathbf u}_k \equiv \eta^{(k)}(0)/k!.$$
Note that a tuple
$$({\mathbf u}_k,{\mathbf u}_{k+1},{\mathbf v}_0,q(x),s(x))$$
forms a set of analytic coordinates in the real slice of $C^\omega(D)\times C^\omega(E)$.
In these coordinates,
\begin{eqnarray}
\nonumber F_1({\mathbf u}_k, {\mathbf u}_{k+1}, {\mathbf v}_0; q,s) &\equiv& \eta(\xi(0))-\xi(\eta(0))={\mathbf u}_k {\mathbf v}_0^k + {\mathbf u}_{k+1} {\mathbf v}_0^{k+1} +q({\mathbf v}_0)-{\mathbf v}_0-s({\mathbf u}_0), \\
\nonumber F_2({\mathbf u}_k, {\mathbf u}_{k+1}, {\mathbf v}_0; q,s) &\equiv&  (\eta \circ \xi)'''(0) - (\xi \circ \eta)'''(0) = \eta'({\mathbf v}_0) \xi'''(0)-\xi'({\mathbf u}_0) \eta'''(0)= \\
\nonumber &=& (k {\mathbf u}_k {\mathbf v}_0^{k-1}+(k+1) {\mathbf u}_{k+1} {\mathbf v}_0^k + q'({\mathbf v}_0)) \xi'''(0) - \\
\nonumber &\phantom{=}&\xi'({\mathbf u}_0) (k(k-1)(k-2) {\mathbf u}_k {\mathbf v}_0^{k-3}+(k+1)(k-1)k {\mathbf u}_{k+1} {\mathbf v}_0^{k-2} +q'''({\mathbf v}_0))\\
\nonumber F_3({\mathbf u}_k, {\mathbf u}_{k+1},{\mathbf v}_0;q,s) &\equiv& {\mathbf v}_0-1.
\end{eqnarray}

We have that 
\begin{eqnarray}
\nonumber D_{{\mathbf u}_k,{\mathbf u}_{k+1},{\mathbf v}_0} {\bf F} &\equiv&  \left[ \begin{array}{ccc}
\frac{\partial F_1}{\partial {\mathbf u}_k}  &  \frac{\partial F_1}{\partial {\mathbf u}_{k+1}}  &  \frac{\partial F_1}{\partial {\mathbf v}_0}\\  \frac{\partial F_2}{\partial {\mathbf u}_k}  &  \frac{\partial F_2}{\partial {\mathbf u}_{k+1}}  &  \frac{\partial F_2 }{\partial {\mathbf v}_0} \\  \frac{\partial F_3}{\partial {\mathbf u}_k}  &  \frac{\partial F_3}{\partial {\mathbf u}_{k+1}}  &  \frac{\partial F_3}{\partial {\mathbf v}_0}\end{array}\right]=\left[ 
\begin{array}{c c c} 
{\mathbf v}_0^k & {\mathbf v}_0^{k+1} & \cdots\\
   {k {\mathbf v}_0^{k-1} \xi'''(0) - \atop \xi'({\mathbf u}_0) k(k-1)(k-2) {\mathbf v}_0^{k-3}}  & {(k+1) {\mathbf v}_0^k  \xi'''(0)- \atop \xi'({\mathbf u}_0)(k+1)(k-1)k {\mathbf v}_0^{k-2} } & \cdots\\
   0 & 0 & 1
\end{array}
\right] 
\\[6pt]
\nonumber &\implies& \operatorname{det}(D_{{\mathbf u}_k,{\mathbf u}_{k+1},{\mathbf v}_0} {\bf F}({\mathbf u}_k,{\mathbf u}_{k+1},1;q,s))=\xi'''(0)-3 k (k-1) \xi'({\mathbf u}_0).
\end{eqnarray}
Let $\zeta_0=(\eta_0,\xi_0)\in\bB^{D,E}$. Then there exists a neighborhood $\cU(\zeta_0)\subset\cW$ in which $|\xi'''(0)|$ is bounded from above and $\xi'(\eta_0)$ is bounded away from zero. Hence, there exists $k\geq 3$ such that in $\cU(\zeta_0)$ the above determinant is non-zero. By Regular Value Theorem this implies the desired result.


\end{proof}

We will denote ${\frak B}^{D,E}$ the complex Banach manifolds
of pairs defined in the same way as $\bB^{D,E}$, but without the condition of real symmetry, so that
$$\bB^{D,E}=({\frak B}^{D,E})^\RR.$$

Our first statement is:
\begin{prop}
\label{invariance1}
The space $\bB$ is renormalization invariant: let $\zeta\in \bB$ and $\rho(\zeta)\neq 0$.
Then $\cR(\zeta)\in \bB$. Moreover, let $\rho(\zeta)\notin \QQ$. Then 
$$\cR^n(\zeta)\to\cA$$
at a geometric rate, where $\cA$ is the hyperbolic horseshoe attractor of renormalization constructed in Theorem~\ref{existence of attractor}.

\end{prop}

\begin{proof}
The space of $C^3$-smooth commuting pairs is $\cR$-invariant, and the geometric convergence statement holds on this space
(see \cite{dF1,dF2,dFdM1}). Preservation of the other properties of pairs in  $\bB$ is evident from the definition of $\cR$.
\end{proof}

\subsection{Complex {\it a priori} bounds}
\begin{defn}
For $0<\mu<1$ and $K>1$ let us denote $\bH(\mu,K)$ the set of almost commuting pairs with the following properties:
\begin{itemize}
\item there exist topological disks $U$, $V$ and $\Delta$ which contain the origin and such that $U$ and $V$ are compactly contained 
in $\Delta$ and 
$$\eta:U\to (\Delta\setminus \RR)\cup \eta(U\cap \RR)\text{ and }\xi:V\to (\Delta\setminus \RR)\cup \xi(V\cap \RR)$$
are three-fold branched coverings;
\item let $A$ be the maximal annulus separating $\CC\setminus \bar\Delta$ from $U\cup V$. Then $\mod A>\mu$;
\item  $\xi(0)=1$ and  $\mu<\eta(0)<1/\mu$;
\item $[0,1]\subset U$ and $[\eta(0),0]\subset V$;
\item $\diam(\Delta)<1/\mu$ and $\Delta$ is a $K$-quasidisk.
\end{itemize}
\end{defn}

\begin{lem}[Lemma 2.15 \cite{Ya2}]
\label{bounds compactness}
For each $\mu>0$ the space $\hol(\mu,K)$ is sequentially pre-compact in the Carath{\'e}odory topology, with every limit point
contained in $\hol(\mu/2,2K)$.
\end{lem}

\begin{thm}
\label{complex bounds}There exists universal constants $\mu>0$ and $K>1$ such that
the following holds.
 Let $S\subset \bB$ be a compact subset. Then there exists $N=N(S)$ such that for every almost commuting pair
 $\zeta\in S$ which is  $n\geq N$ times renormalizable,
the renormalization
$$\zeta_n=R^n\zeta\in \bH(\mu,K).$$
Furthermore, there exists a universal $R>1$ such that the range $\Delta_n$ of $\zeta_n$ can be chosen as 
$\Delta_n=D_R(0).$
\end{thm}

\realfig{domain}{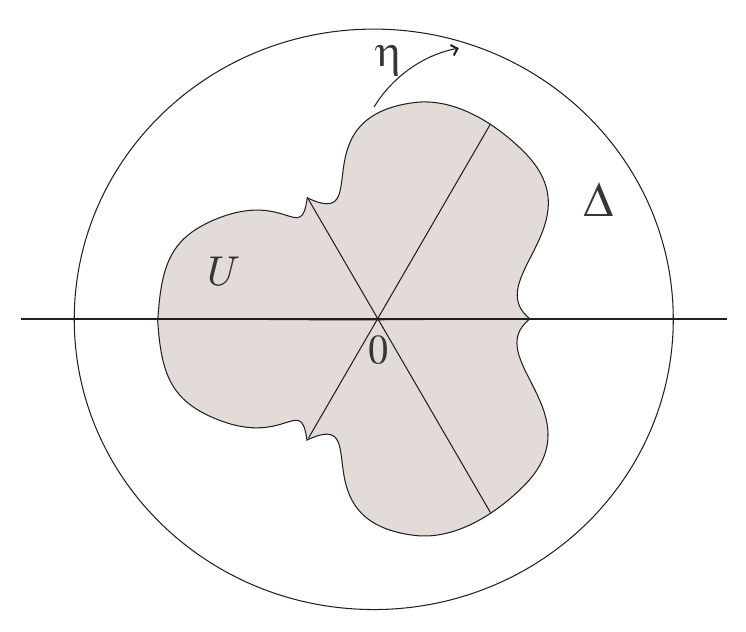}{An extension of $\eta$ as a $3$-fold branched covering map $U\to\Delta$. The preimage of the real line is indicated.}{10cm}

The proof of this theorem was first given by the second author in \cite{Ya1} for $C^\omega$-commuting pairs in the Epstein class, and 
was later adapted in \cite{dFdM2} for $C^\omega$-commuting pairs without the Epstein property. However, these arguments do not
use commutativity of the pair beyond order zero (i.e. $\eta\circ\xi(0)=\xi\circ\eta(0)$). Hence, the theorem holds in the above generality.

\noindent

We conclude this section with the following statement
which is an immediate consequence of \thmref{complex bounds} and the compactness statement of \lemref{bounds compactness}:
\begin{thm}
\label{invariance2}
There exists a space
$\bB^{D,E}$ and $m\in\NN$ such that the following holds. Let $\zeta\in \bB^{D,E}$ be an $m$-times renormalizable almost commuting pair.
 There exist larger domains $D'\Supset D$, and $E'\Supset E$ so that
$$\cR^m(\zeta)\in  \bB^{D',E'}.$$
\end{thm}
\begin{proof}
Let $0<\mu<1$, $K>1$ be as in \thmref{complex bounds}. By  \lemref{bounds compactness}, there exist $\bB^{D',E'}$ such that $\hol(\mu,K)\subset \bB^{D',E'}$. Let us fix $\bB^{D,E}$ so that $D'\Supset D$, and $E'\Supset E$. 
By Koebe Distortion Theorem, the set $\bB^{D,E}$ is compact in $\bB$ in the $C^3$-metric on the real line. This implies that the constant $N$ in \thmref{complex bounds}
can be chosen uniformly in $\bB^{D,E}$. To complete the proof, let $m\geq N$.
\end{proof}

\section{Hyperbolicity of renormalization in one dimension}
This section is devoted to proving the following theorem:
\begin{thm} \label{th:main1}
Let us fix a periodic point $\zeta_*\in\cA$ of $\cR$ of period $k$ and let $\rho_*=\rho(\zeta_*)$.  There exists a space
$\bB^{D,E}$ and $p=m\cdot k\in\NN$ such that the following holds.  The pair $\zeta_*$ is a  fixed point of $\cR^p$ in the space $\bB^{D,E}$. 
The image
$$\cR^p(\zeta_*)\in  \bB^{D',E'} \text{ where }D'\Supset D,\;E'\Supset E.$$
The linearization
$$\cL\equiv D\cR^p|_{\zeta_*}$$
in  $\bB^{D,E}$ is a compact operator with one simple unstable eigenvalue, 
and the rest of the spectrum is compactly contained in $\DD$.
The stable manifold $\cW^s(\zeta_*)$ of $\zeta_*$ contains all pairs in $\bB^{D,E}$ with the rotation number $\rho_*$.
\end{thm}

\noindent
Let $\zeta \in \cW^s(\zeta_*)$ and  consider its $n$-th pre-renormalization $\zeta_n=(\zeta^{\bar s_n},\zeta^{\bar t_n})$ defined on linear rescalings $D_n$ and $E_n$ of the sets $D$ and $E$ correspondingly. Consider the collection of topological disks
$$\cV_n\equiv \{\zeta^{\bar w}(D_n)\text{ for all }\bar w\prec \bar s_n\text{ and }\zeta^{\bar w}(E_n)\text{ for all }\bar w\prec \bar t_n\}.$$
We will refer to this collection of sets the {\it $n$-th complex dynamical partition} of $\zeta$. It is clear from the  construction that the elements  $\zeta^{\bar w}(I_n)$ and $\zeta^{\bar w}(J_n)$ of the dynamical partition $\cP_n$ are contained in the elements $\zeta^{\bar w}(D_n)$ and $\zeta^{\bar w}(E_n)$, repectively, of the complex dynamical partition $\cV_n$.
Set $\lambda_n=(-1)^n|I_n|$ so that 
$$\cR^n\zeta(z)=\lambda_n^{-1} p\cR^n\zeta(\lambda_n z).$$

As a consequence of Theorem $\ref{th:main1}$ we have the following: 

\begin{cor} \label{cor:partition}
Let $\zeta_*$  be  as in Theorem $\ref{th:main1}$. Let $\zeta \in \cW^s(\zeta_*)$. Then there exists $N=N(\zeta)$, $C>0$, $C'>0$, $K>0$  and $0<\gamma<1$ so that for every $n>N$  the following holds.

\begin{itemize}
\item[$1)$] If $Q_n \in \cV_n$ then $\diam(Q_n)< C \gamma^n$.
\item[$2)$] Any two neighboring domains $Q_n, Q'_n \in \cV_n$ are $K$-commensurate.
\item[$3)$] For every $\bar w \prec \bar s_n$ (or $\bar w \prec \bar t_n$)  set $\psi_{\bar w}^\zeta=\zeta^{\bar w} \lambda_n$. Then $\|D \psi_{\bar w}^\zeta |_D \|_\infty<\gamma^n$ ($\|D \psi_{\bar w}^\zeta |_E \|_\infty<\gamma^n$, respectively).

\end{itemize}
\end{cor}
\begin{proof}
By \thmref{th:main1}, there exists $N>0$ and a pair of domains $\hat D\Supset D$ and $\hat E\Supset E$  such that for all $n\geq N$ the 
maps of the pair $\cR^n\zeta\in\bB^{\hat D,\hat E}.$
By Koebe Distortion Theorem, this implies that for all $\bar w \prec \bar s_n$ (or $\bar w \prec \bar t_n$) the branches $\zeta^{-\bar w}$ have 
bounded distortion. The claims readily follow.
\end{proof}

\subsection{Expansion of renormalization}

In this section we will describe the expanding direction  of renormalization. For the remainder of this chapter, let us fix the domains
$D$, and $E$ as in \thmref{invariance2}. 

\subsection*{Definition of the expanding cone field}
We begin by defining a subset $\cC$ in the tangent bundle  $\bT\equiv T\bB^{D,E}$ as follows. Let $\bar v(x)\in\bT_\zeta$ for some
renormalizable pair $\zeta$. 
Let $\zeta$ be a twice renormalizable pair, and recall that $p\cR^2\zeta$ denotes the second pre-renormalization (the non-rescaled iterate)
of $\zeta$. 
Denote 
$$\cC_\zeta=\{\bar v\in\bT\;|\; \inf_x\nabla_{\bar v}p\cR^2\zeta>0\text{ for all }x\in I_2\cup J_2\},$$
(where $\nabla_{\bar v}$ denotes the directional derivative in the direction of $\bar v$) and set $\cC=\cup\cC_\zeta$ over all twice-renormalizable pairs $\zeta\in\bB^{D,E}$.

\begin{prop}
\label{properties cone}
For every twice-renormalizable $\zeta$,  the set $\cC_\zeta$ is an open cone in $\bT_\zeta$.
\end{prop}

We next prove:
\begin{prop}
\label{rho-monotone}
Let $\zeta(t):(0,1)\to\bB^{D,E}$ be a smooth curve with the property 
$$\frac{d}{dt}\bar \zeta(t)\in \cC_{\zeta(t)}\text{ for all }t.$$
Then the function
$$\rho(t)\equiv \rho(\zeta(t))$$
is non-decreasing. Furthermore, if $\rho(t_0)\notin \QQ$ then $\rho(t)$ is strictly increasing at $t_0$.
\end{prop}
\begin{proof}
Fix $t_0\in(0,1)$ and let $\zeta({t_0})^k(0)\neq 0$ 
be a closest return of $0$
under the dynamics of the pair $\zeta({t_0})$. 
An easy induction based on the Chain Rule shows that $\frac{d}{dt}\zeta(t)^k(0)|_{t=t_0}$ is positive for all $k$ starting with the first returns corresponding to the second renormalization. Thus, the 
heights $r_{2i}$ of renormalizations $\cR^{2i}\zeta(t)$ {\it decrease}, and the heights $r_{2i+1}$ of renormalizations $\cR^{2i+1}\zeta(t)$
{\it increase} with $t$. Hence, the value of the rotation number $\rho=[r_0,r_1,\ldots]$ is a non-decreasing function of $t$.
The last assertion is similarly evident and is left to the reader. 
\end{proof}

\subsection*{The expansion properties of the cone field $\cC$.}
We begin by recalling
how the composition operator acts on vector fields. For a pair of analytic functions $f$ and $g$ of the real variable, denote 
$$\text{Comp}(f,g)=f\circ g.$$
Consider $\text{Comp}$ as an operator $C^\omega\times C^\omega\to C^\omega$ and let $D\text{Comp}$ denote its differential. An elementary calculation shows that 
\begin{equation}
\label{eqn-comp}
D\text{Comp}|_{(f,g)}:(\phi,\gamma)\to f'\circ g\cdot \gamma+\phi\circ g.
\end{equation}
The significance of the formula (\ref{eqn-comp}) for us lies in the following trivial observation: if $f$ and $g$ are both increasing functions,
and the vector fields $\phi$ and $\gamma$ are non-negative, then
\begin{equation}
\label{eqn-comp2}
\inf_x D\text{Comp}|_{(f,g)}(\phi,\gamma)\geq \inf_x\phi.
\end{equation}

\begin{prop}
\label{pr:non-empty}
Fix a twice-renormalizable pair $\zeta=(\eta,\xi)\in \bB^{D,E}$. Then $\cC_\zeta$ is non-empty.

\end{prop}
\begin{proof}
 Let $\bar v=(\bar\alpha,\bar\beta)$ have the properties:
\begin{itemize}
\item $\bar \alpha(x)>0$, $\bar\beta(x)>0$ for real $x$ such that $x\notin\{0,1,\eta(0)\}$;
\item for each $x\in\{0,1,\eta(0)\}$, the vector field $\bar v(x)$ vanishes to order $3$.
\end{itemize}
It is evident that vector fields with these properties exist (they can be taken to be polynomial, for instance), and that every such
$\bar v\in \bT_\zeta$. Finally, $\bar v\in \cC_\zeta$ by the Chain Rule (\ref{eqn-comp}).
\end{proof}

For a renormalizable pair $\zeta=(\eta,\xi)$  let us set 
$$\lambda_{\zeta}=\eta^{r_0}(1)>0,$$
where, as before, $r_i$ denotes the height of $\cR^i\zeta$.
%

\begin{prop}
\label{expcone1}
There exist $k\in\NN$ and $\delta>0$ such that the following holds. Let $\zeta\in\bB^{D,E}$ and let $\bar v\in\cC_\zeta$.
Then
$$||D\cR^{2k}_\zeta\bar v||>C\cdot\eps (1+\delta)^k,$$
where $C$ is bounded on compact subsets of $\bB^{D,E}$ and $\eps=\inf Dp\cR^2\bar v(x)>0$.
\end{prop}
\begin{proof}
Let $\bar v(x)=(\bar\alpha(x),\bar\beta(x))\in \cC_{\zeta}$. Consider a smooth deformation 
\begin{equation}
\label{deform}
\zeta_t^{\bar v}=(\eta+t\bar\alpha+o(t),\xi+t\bar\beta+o(t))\equiv (\eta_t,\xi_t)\in\bB^{D,E}.
\end{equation}
For $m\in\NN$ let us denote  $$\cR^{2m}\zeta_t^{\bar v}\equiv (\eta_{t,m},\xi_{t,m})\text{, and }p\cR^{2m}\zeta_t^{\bar v}\equiv(H_{t,m},K_{t,m}).$$
Let $$\lambda_{t,m}\equiv K_{t,m}(0)>0.$$
An easy induction shows that 
\begin{itemize}
\item[(a)] $\eta_{t,k}(x)=\frac{1}{\lambda_{t,k}} H_{t,k}\circ (\lambda_{t,k} x);$
\item[(b)] $H_{t,k}(0)<0$.
\end{itemize}
A repeated application of (\ref{eqn-comp}) 
implies that 
\begin{itemize}
\item[(c)] $\frac{\partial}{\partial t}H_{t,k}(x)>\eps>0$ where $\eps=\inf Dp\cR^2\zeta\bar v(x)$;
\item[(d)] $\frac{d}{d t}\lambda_{t,k}>0$.
\end{itemize}
We calculate:
$$\frac{\partial}{\partial t}\left( \frac{1}{\lambda_{t,k}}H_{t,k}(\lambda_{t,k} x) \right)=-\frac{\frac {d}{dt}\lambda_{t,k}}{(\lambda_{t,k})^2}H_{t,k}(\lambda_{t,k} x)+\frac{1}{\lambda_{t,k}}\left(\frac{\partial H_{t,k}(\lambda_{t,k} x)}{\partial t}+\frac{\partial H_{t,k}(x)}{\partial x}\frac{d \lambda_{t,k}}{d t}x\right).$$
Substituting $x=0$ and using $(a)-(d)$ we see that  
$$\left. \frac{\partial}{\partial t}\right| _{t=0}\left. \left( \frac{1}{\lambda_{t,k}}H_{t,k}(\lambda_{t,k} x) \right)\right| _{x=0}=
D\cR^{2k}\bar v(0)\geq \frac{1}{\lambda_{0,k}}\eps.$$
The standard real {\it a priori} bounds imply that 
$$\lambda_{0,k}\leq C(1+\delta)^{-k},$$
where $\delta>0$ is universal, and $C$ is bounded on compact subsets of $C^3$-commuting pairs,
which completes the proof.
\end{proof}

\subsection{Local stable manifold of a periodic point of $\cR$}
\label{stable manifold section}

\noindent
As before, let us work in the notation of \thmref{th:main1}. Set $\zeta\equiv\zeta_*$.
 
Set $\rho=\rho(\zeta)$, and define $$\cD_\rho=\{\gamma\in\bB^{D,E},\text{ such that }\rho(\gamma)=\rho\}.$$

The following proposition 
 directly follows from \thmref{thmGuarino} and compactness considerations:
\begin{prop}
\label{local stable}
There exists a neighborhood $Y$ of $\zeta$ in $\bB^{D,E}$ such that for every 
$\gamma\in Y\cap \cD_\rho$
$$\cR^{pm} \gamma\underset{m\to\infty}{\longrightarrow}\zeta$$ at a geometric rate, uniformly in $Y$.
\end{prop}

Below we shall demonstrate that the local stable set of $\zeta$ is a graph over a hyperplane:
\begin{thm}
\label{stable manifold}
There is an open neighborhood $W\subset \bB^{D,E}$ of $\zeta$ such that $\cD_\rho\cap W$
is a $C^0$-graph over a hyperplane in a local chart in $\bB^{D,E}$.
\end{thm}

\noindent
Denote $p_k/q_k$ the reduced form of the $k$-th
continued fraction convergent of $\rho$. Furthermore, 
define $\cD_k$ as the set of $\gamma\in\bB^{D,E}$ for which $\rho(\gamma)=p_k/q_k$ and 
$0$ is a periodic point with period $q_k$.
As follows from the Implicit Function Theorem, this is a local codimension $1$ submanifold.
We note:
\begin{lem}
\label{cone not in tk}
Let $\gamma\in \cD_k$ for $k=2m\geq 2$, and denote $T_\gamma\cD_k\subset \bT_\gamma$ the tangent space to $\cD_k$
 at this point.
Then $$T_\gamma\cD_k\cap \cC_\gamma=\emptyset.$$
\end{lem}
\begin{pf}
Let $\bar v\in \cC_\gamma$ and suppose $\{\gamma_t\}$ is a one-parameter family
such that  $$\gamma_t=\gamma+t\bar v+o(t).$$
Then for sufficiently small values of $t>0$, 
$\gamma_t^{q_k}>\gamma^{q_k}$, and  hence $\gamma_t^{q_k}(0)\neq 0.$
\end{pf}

\noindent
Now let $\bar v\in \cC_\zeta$ be as in the proof of \propref{pr:non-empty},
 and let
$\{\zeta_t\}$ be a one-parameter family in $\bB^{D,E}$
such that  $$\zeta_t=\zeta+t\bar v.$$
Elementary considerations of the Intermediate Value Theorem imply that 
for every large enough $k$  there exists a value of $t>0$ such that 
the map $\zeta_t\in\cD_k$. 
Moreover,
if we denote $t_k$ the smallest parameter with this property, 
then $t_k\to 0$. 
Set $\zeta_k=\zeta_{t_k}$ and let $T_k=T_{\zeta_k}\cD_k\subset \bT$.
By  \lemref{cone not in tk} and the Hahn-Banach Theorem there exists $\eps>0$
such that for every $k$ there exists a linear functional $h_k\in (\bT_\zeta)^*$
with $||h_k||=1$,
such that $\operatorname{Ker}h_k=T_k$ and $h_k(\bar v)>\eps$.
By the Alaoglu Theorem, we may select a subsequence
$h_{n_k}$ weakly-$*$ converging to $h\in (\bT_\zeta)^*$. Necessarily 
$\bar v\notin \operatorname{Ker}h$, so
$h\not\equiv 0$. Set $T=\operatorname{Ker}h$.

\smallskip
\noindent
{\it Proof of \thmref{stable manifold}.} 
By the above, we may select a splitting $\bT_\zeta=T\oplus \bar v\cdot\RR$.
Denote $p:\bT_\zeta\to T$ the corresponding projection, and let
$\Psi:\bB^{D,E}\to \bT_\zeta$ be a local chart at $\zeta$.
\lemref{cone not in tk} together with the Intermediate Value Theorem imply that
$p\circ \Psi:\cD_k\to T$ is an isomorphism onto the image,
and there exists an open neighborhood $\cU$ of the origin in $T$, such that 
$p\circ \Psi(\cD_k)\supset \cU$. 
We may
select a $C^0$-converging subsequence $\cD_{k_j}$,
whose limit is a graph $G$ over $\cU$.
 Necessarily, for every $\gamma\in G$, $\rho(\gamma)=\rho$. As we have seen above,
every point $\gamma\in\cD_\rho$ in a sufficiently small neighborhood of $\zeta$ is in
$G$, and thus $G$ is an open neighborhood in $\cD_\rho$.
$\Box$

\subsection{Proof of \thmref{th:main1}}
\label{sec:hyperb}
Let us work in the notation of \thmref{th:main1} again. Note that by \thmref{invariance2}, the operator $\cL$ is compact, and hence, by the standard facts of the spectral theory of compact operators, we have:
\begin{itemize}
\item every element of the spectrum of $\cL$ is an eigenvalue;
\item the spectrum of $\cL$ has no accumulation points except for $0$.
\end{itemize}
We now prove:
\begin{prop}
\label{hyperb2}
The operator $\cL$ has a single unstable eigenvalue.
\end{prop}
\begin{proof}
By \propref{expcone1}, the operator $\cL$ has at least one unstable eigenvalue. On the other hand, by \thmref{stable manifold},
$\operatorname{dim} W^u(\zeta)<2$. 
\end{proof}
Finally, 
\begin{prop}
\label{hyperb1}
The operator $\cL$ has no eigenvalues on the unit circle.
\end{prop}
\begin{proof}
  Assume the contrary. By the spectral decomposition properties of the compact operator $\cL$, the tangent space decomposes into an $\cL$-ivariant direct sum $E^u\oplus E^c\oplus E^s$, where $E^u$ is the one-dimensional unstable eigenspace, $E^c$ is a finite-dimensional union of eigenspaces corresponding to neutral eigenvalues, and $E^s$ is the strong stable space of a finite codimension. The standard Central Manifold Theorem considerations can now be applied to $\cR^p$ at $\zeta_*$ (see e.g. \cite{BV} for the infinite-dimensional setting), which imply that
there exists a finite-dimensional smooth central manifold $W^c$ 
at $\zeta$. Now, $W^c$  is transverse simultaneously to $D_\rho$ at $\zeta$ and to the cone $\cC_\zeta$. This is clearly impossible by dimensionality considerations.

\end{proof}

\section{Extending renormalization to dissipative two-dimensional pairs} \label{sec:RenACM}

\subsection{Some function spaces}
Let $\zeta_*=(\eta_*,\xi_*)\in\bB^{D,E}$ be the hyperbolic fixed point of $\cR^p$ constructed in  \thmref{th:main1}.
We denote
$$\bC^{D,E}\equiv (C^\omega(D)\times C^\omega(E))^\RR\supset\bB^{D,E}.$$
%
We set $\Omega=D \times D$, $\Gamma=E \times E$, and  let $\bU^{\Omega, \Gamma}$ to be the space of pairs of maps
$$A:\Omega\to\CC^2,\; B:\Gamma\to \CC^2,$$
where $A$ and $B$ are both analytic and continuous up to the boundary, equipped with the norm
$$||(A,B)||=\frac{1}{2}(||A||+||B||),\text{ where }||.||\text{ stands for the uniform norm.}$$

For convenience, for a smooth function $F$ from a domain $W\subset\CC^2$ to $\CC^2$, we will adopt the notation
$$||F||_y=\underset{(x,y)\in W}{\operatorname{sup}}||\partial_yF(x,y)||.$$

We set
$$\bD^{\Omega, \Gamma}\equiv (\bU^{\Omega, \Gamma})^\RR;$$
so that $\bD^{\Omega, \Gamma}$ consists of pairs of real-symmetric two-dimensional maps. 
%
%
%
Let us define a ``diagonal'' isometric embedding $\iota$ of  the manifold $\bC^{D,E}$ into $\bD^{\Omega,\Gamma}$, which sends a pair $\zeta=(\eta,\xi)$ to
a pair of functions $\iota(\zeta)$ given by
$$\left(\left( x \atop y\right)\mapsto \left( \eta(x) \atop \eta(x)  \right), \left( x \atop y\right)\mapsto \left( \xi(x) \atop \xi(x)  \right)\right).$$
Let us denote $\pi_1$ and $\pi_2$ the two coordinate projections $\CC^2\to\CC$.
For a pair of two-dimensional maps $(A,B)(x,y)$ let us define 
$$\cL (A,B)(x,y)\equiv (\pi_1(A(x,0)),\pi_1(B(x,0)))=(a(x,0),b(x,0)).$$
In this way, we have
$$\cL\circ \iota\equiv \text{Id}.$$

The action of renormalization operator $\cR$ naturally extends to the ``diagonal'' subspace
$\iota(\bB^{D,E})$ as
$$\hat \cR\equiv \iota\circ \cR\circ\iota^{-1}.$$
Our goal is to further extend it to an analytic operator acting on a
small neighbourhood of this subspace in the space of two-dimensional maps.
For a choice of 
$\delta>0$, and $\eps>0$ (where we should think of $\eps$ as being much smaller than $\delta$),
we let $\bB^{D,E}_\delta$ denote a $\delta$-neighborhood of $\zeta_*$, and let
$\bD^{\Omega,\Gamma}_{\eps,\delta}$ be the $\eps$-neighborhood of $\iota(\bB_\delta^{D,E})$ in $\bD^{\Omega,\Gamma}$. In other words,
  a pair of maps $(A,B)$ in $\bD^{\Omega,\Gamma}_{\eps,\delta}$ has the form:
\begin{eqnarray}
\label{eq:AA} A(x,y)&=&(a(x,y),h(x,y))=(a_y(x),h_y(x)),\\
\label{eq:BB} B(x,y)&=&(b(x,y),g(x,y))=(b_y(x),g_y(x)), 
\end{eqnarray}  
where $a_y(x)$ and $h_y(x)$ are $\eps$-close to $\eta(x)$, and $b_y(x)$ and $g_y(x)$ are  $\eps$-close to $\xi(x)$ for all values of $y$,
where $(\eta,\xi)\in \bB^{D,E}_\delta$.

In what follows, we will demonstrate that there exists  $\eps>0$, and $n_0\in\NN$ such that for every $n\geq n_0$ which is the multiple of $p$, the operator $\hat \cR^n$ extends to an analytic operator defined in $\bD^{\Omega,\Gamma}_{\eps,\delta}$ which has the same hyperbolic properties as the one-dimensional version. The definition of this extension to two-dimensional perturbations is somewhat involved. In brief, it consists of the following steps:
\begin{enumerate}
\item pre-renormalization will now be defined not in a neighborhood of the ``critical point'' $(0,0)$ but in the neighborhood of the point $(\eta^{-1}(0),0)$, where $\eta=a_0(x).$ It is then pulled back to the neighborhood of the origin by a non-linear coordinate change, which is a small perturbation of $\eta$. This results in:
\item reduction of the order of the perturbation: similarly to \cite{dCLM}, the non-linearly rescaled pre-renormalization is in the $\eps^2$-neighborhood of the diagonal subspace $\iota(\bC^{D,E})$. However, it does not have a well-defined projection onto an element  of $\iota(\bB^{D,E})$, which we further rectify:
  \item by defining a projection from general two-dimensional pairs $(A,B)$ onto pairs $(\tl A,\tl B)$ such that $\cL(\tl A,\tl B)\in \bB^{D,E}.$ This projection is not dynamical, however, crucially for our applications, it does not affect the pairs $(A,B)$ which actually commute; in particular, when $A=F^{q_{n+1}}$ and $B=F^{q_n}$ are iterates of the same map.

 \end{enumerate}
\noindent
We now proceeed with the construction.

\subsection{Definition of pre-renormalization and non-linear change of coordinates}



Let $n\geq 3$ be a multiple of $p$, and 
let $\zeta \in \bB^{D,E}$ be $n$-times renormalizable,
$$\cR^n\zeta= \lambda_n^{-1}\circ \left(\zeta^{\bar s_n}, \zeta^{\bar t_n } \right) \circ \lambda_n,$$
where $\bar s_n$ and $\bar t_n$ are as in (\ref{sntn}).

Let $U_1\Supset U_2\Supset (D\cup E)$ be two compactly nested topologicals disks, the smaller of which compactly containes the union of the domains of definition of the elements of $\zeta_*$.

In what follows, 
we fix $n=p\cdot k\geq 2$, $\delta>0$ in such a way that for all $\zeta\in \bB_{2\delta}^{D,E}$, we have:
\begin{itemize}
\item the function $\eta^{-1}$ is a diffeomorphism of $\lambda_n(U_1)$ onto its image (which is a neighborhood of $\eta^{-1}(0)$.

\end{itemize}
\noindent

Let $\bar s_n=(a_1,b_1,\ldots,a_{m_n},b_{m_n})$, note that $b_{m_n}=0$, and denote 
\begin{eqnarray}
\nonumber \h s_n&=&\left\{ (a_1,b_1,a_2,b_2,\ldots,a_{m_n}-2,0), \ a_{m_n} \ge 2  \atop (a_1,b_1,a_2,b_2,\ldots,b_{m_{n-1}}-1,0,0), \ a_{m_n}=1 \right. ,\\
\nonumber \phi&=&\left\{ \eta^2, \ a_{m_n} \ge 2  \atop \eta \circ \xi, \ a_{m_n}=1 \right. .
\end{eqnarray}
Define $\h t_n$ in a similar way. Then $\cR^n\zeta$ can be written as
$$\cR^n\zeta=(\lambda_n^{-1}\circ \phi \circ \zeta^{\h s_n} \circ \lambda_n,\lambda_n^{-1}\circ \phi \circ  \zeta^{\h t_n }  \circ \lambda_n).$$

\noindent
Let us apply the diffeomorphic change of coordinates $\eta^{-1}$ to  $p\cR^n\zeta$ to obtain a pre-renormalization in a neighborhood of $\eta^{-1}(0)$:
\begin{equation}\label{preren-crit}
\h p \cR^n\zeta=\left(\eta^{-1} \circ  \zeta^{\bar s_n} \circ \eta, \eta^{-1} \circ \zeta^{\bar t_n } \circ \eta \right)=\left( f \circ   \zeta^{\h s_n} \circ \eta, f \circ \zeta^{\h t_n } \circ \eta \right),
\end{equation}
where 
\begin{equation}
\label{eq:def-f}
f=\eta\text{ if }a_{m_n} \ge 2\text{ and }f=\xi\text{ if }a_{m_n}=1.
\end{equation}


Now, let $\eps<\delta$ and let  \begin{equation}
  \label{Z-eq}
  Z=(A,B)\in\bD^{\Omega,\Gamma}_{\eps,\delta},\text{ and }\zeta=\cL(Z)\in\bB_{2\delta}^{D,E},\; ||Z-\iota(\zeta)||=O(\eps).
  \end{equation}
Set $$\Lambda_n(x,y)\equiv (\lambda_nx,\lambda_ny).$$
In an analogous fashion to (\ref{preren-crit}), we set
\begin{equation}
\label{preren-crit1}
\h p \cR^n Z = \left(F \circ Z^{\h s_n} \circ A, F \circ Z^{\h t_n} \circ A \right),
\end{equation}
where $F=A$ if $a_{m_n} \ge 2$ and $F=B$ if $a_{m_n}=1$.

Let us set
$$\phi_1(x)\equiv \left\{ \pi_1 A^2(x,0)= a(a(x,0),h(x,0)), \ a_{m_n} \ge 2  \atop \pi_1 A \circ B(x,0)= a(b(x,0),g(x,0)), \ a_{m_n}=1 \right.,$$
and
$$f_2(x)\equiv \pi_2 F(x,0)= \left\{ h_0(x), \ a_{m_n} \ge 2  \atop g_0(x), \ a_{m_n}=1 \right.$$
We now define a pair of maps: 
\begin{equation}
\label{wvtransform}V(x,y) :=  \left(\begin{array}{c}a_y(x)\\y\end{array}\right)\text{, and }W(x,y) := \left(\begin{array}{c}x\\ \phi_1(f_2^{-1}(y)) \end{array}\right),
\end{equation}
Let us set
\begin{equation}
  \label{coordh}
  H\equiv W\circ V.
\end{equation}
  By considerations of continuity we immediately have:
\begin{prop}
  \label{choiceeps1}
  there exists $\eps_1>0\in(0,2\delta)$ such that for every $\eps\in(0,\eps_1)$, the map $H$ as defined above is a diffeomorphism of
  $\eta_*^{-1}(\lambda_n(U_2))\times \eta_*^{-1}(\lambda_n(U_2))$ onto its image.
\end{prop}
Observe that 
$$A\circ V^{-1}(x,y)=\left(\begin{array}{c}x\\ h(a_y^{-1}(x),y) \end{array}\right),$$
and hence
\begin{equation}
  \label{boundv}
||A\circ V^{-1}||_y=O(\eps).
\end{equation}
Similarly,
\begin{equation}
  \label{boundw}
||W\circ V\circ F-\iota(\phi_1(x))||=O(\eps).
\end{equation}
We  define the $n$-th pre-renormalization of $Z=(A,B)$ as the pair
\begin{equation}
  \label{preren-eq}
  p\cR^nZ=p \cR^n Z=(\bar A, \bar B) = H \circ F \circ \left( Z^{\h s_n},  Z^{\h t_n} \right) \circ A \circ H^{-1}(x,y),
  \end{equation}
and set
\begin{equation}
  \label{lambdan}
  \Lambda_n(x,y)=(\ell_nx,\ell_ny),\text{ where }\ell_n=\pi_1\bar B(0,0).
  \end{equation}
\noindent

\begin{prop}
  \label{lem-preren} There exists $\eps_2\in(0,\eps_1)$ such that for every $\eps\in(0,\eps_2)$ the following holds.
  The pre-renormalization $p\cR^n(Z)$ is a pair of analytic mappings defined in domains $\Lambda_n(\Omega),\;\Lambda_n(\Gamma)$ respectively, such that
  \begin{equation}
    \label{prereneq1}
 ||p\cR^n(Z)-\iota(p\cR^n\zeta)||=O(\eps),
  \end{equation}
 where $\zeta$ is as in (\ref{Z-eq}),  and
  \begin{equation}
    \label{prereneq2}
||p\cR^n(Z)||_y=O(\eps^2),
  \end{equation}
in these domains.
\end{prop}
\begin{proof}
  The bound (\ref{prereneq1}) follows for all sufficiently small $\eps$ from (\ref{boundw}) and straightforward continuity considerations.
  To obtain the second bound, note that by  (\ref{boundv}), and since the matrix $DW$ is diagonal, the differential
  $D(A\circ H^{-1})$ has the form
  $$D(A\circ H^{-1})=\left[ \begin{array}{cc} O(1) &0 \\ O(1)& O(\eps) \end{array}\right] .$$
  The differential of the remainder of the composition is (since it is an $\eps$-small perturbation of a ``diagonal'' function of $x$) of the form:
  $$\left[ \begin{array}{cc} O(1) & O(\eps) \\ O(1)& O(\eps) \end{array}\right].$$
  The claim immediately follows.
  \end{proof}

\subsection{Projection on the space of  almost commuting pairs}
Let us set $$\tilde Z= (\tilde A, \tilde B)\equiv \Lambda_n^{-1}\circ p\cR^n Z\Lambda_n.$$
In view of the above, it is a small (of order $\eps^2$) perturbation of the ``diagonal'' pair $\iota\circ\tl\zeta$, where
$$\tl\zeta=(\tl\eta,\tl\xi)\equiv \cL(\tilde Z).$$
There is, of course, no reason for the almost commutation condition to hold for $\tl\zeta$. This would create new ustable directions for  renormalization,
so our next step is to define a {\it projection} which imposes
such a condition onto almost diagonal pairs.

To that end, we set 
$$\Pi (\tilde A,\tilde B)(x,y)=(\tl A,\tl B)+\left( \left( a x^4+b x^6  \atop   \tilde a x^4+b x^6\right), \left( c + d x + e x^2  \atop  c+d x +e x^2   \right)   \right),$$
and require that the pair $(\hat A,\hat B)\equiv \Pi (\tilde A,\tilde B)(x,y)$ satisfies the following two-dimensional version of almost commutation conditions:
\begin{eqnarray} \label{commutation}
\pi_1(\hat A \circ  \hat B(x,0) -  \hat B \circ  \hat A(x,0))&=&o(|x|^3),\\
\label{normalization} \pi_1 \hat B(0,0)&=&1.
\end{eqnarray}
We claim:
\begin{prop}\label{prop:2Dprojection}

  There exist $\eps_3\in(0,\eps_2)$, $L>0$, such that for all $\eps\in(0,\eps_3)$ the following holds. For every pair
  $(\tl A,\tl B)\in \bD^{\Omega,\Gamma}_{\eps,\delta}$ there exists a unique tuple
 $(a,b,c,d,e)\in \DD_{L \eps^2}(0)^{\otimes 5}$ such that the conditions (\ref{commutation})-(\ref{normalization}) hold. Furthermore, the 
map $$(\tl A, \tl B) \mapsto (a,b,d,e,c)$$ is analytic.

\end{prop}
The proof of \propref{prop:2Dprojection} is carried out by a brute force Regular Value Theorem argument based on calculating the differential of the system of non-linear equations given by the above conditions. To streamline the text, we give it in the Appendix \S~\ref{Apx1} 
For ease of reference, let us note that, by the uniqueness part of the statement of \propref{prop:2Dprojection}:
\begin{prop}
  \label{trivialproj}
  Suppose $\eps\in(0,\eps_3)$ and $(\hat A,\hat B)\in \bD^{\Omega,\Gamma}_{\eps,\delta}$. Assume that the conditions
(\ref{commutation})-(\ref{normalization}) hold for it.
  Then $\Pi(\hat A,\hat B)=(\hat A,\hat B)$.
  \end{prop}

\subsection{Renormalization of two-dimensional pairs}

\noindent
We let $\eps\in(0,\eps_3)$, and define the {\it order $n$ renormalization} of a pair $(A,B) \in   \bD_{\eps,\delta}^{\Omega,\Gamma}$ as
\begin{equation}\label{2Drenorm}
\hat\cR_n(A,B)=\Pi \Lambda^{-1}_n \circ p \cR^n(A,B) \circ \Lambda_n.
\end{equation}



\noindent
By construction, we have:
\begin{thm} \label{thm:renbounds}
  There exists $\eps_4\in(0,\eps_3)$ such that for $\eps\in(0,\eps_4)$, 
  $$\hat\cR_n:\bD_{\eps,\delta}^{\Omega,\Gamma}\to\bD^{\Omega,\Gamma},$$
  and is an analytic operator. Furthermore, 
  $$\hat\cR_n=\iota\circ \cR^n\circ\iota^{-1}$$
  on $\iota(\bB^{D,E}_\delta)$.
  
Additionally, if  $Z\in \bD_{\eps,\delta}^{\Omega,\Gamma}$ does not depend on $y$ then  
$\hat \cR_nZ\in \iota(\bC^{D,E}).$
\end{thm}

\noindent
Denote $Z_*=\iota(\zeta_*)$; it is a fixed point of $\hat\cR_n$. 
We have:
\begin{thm}
\label{hyperbolicity-thm2}
The differential $\cD=D|_{Z_*}\hat\cR_n$ is a compact operator. The non-trivial part of its spectrum corresponds to one-dimensional ``diagonal'' maps:
all of the
eigenspaces corresponding to non-zero eigenvectors lie in the tangent bundle to  $\iota(\bC^{D,E})$.

Its strong stable manifold has codimension at most $3$. Its spectrum coincides with the spectrum of the differential of one-dimensional renormalization  $D(\cR^n|_{\zeta_*})$ plus at most two more   eigenvalues. 
\end{thm}
\begin{proof}  
  By \propref{lem-preren}, for each small $\eps$ and $Z\in \bD^{\Omega,\Gamma}_{\eps,\delta}$, the distance from the rescaled pre-renormalization
  $$\tl Z=\Lambda^{-1}_n \circ p\cR^n\circ\Lambda_n(Z)$$ to $\iota(\bC^{D,E})$ is of the order $\eps^2$. By \propref{prop:2Dprojection} (analyticity of the projection $\Pi$),
  the same holds true for $\hat Z=\Pi \tl Z$. This,  \propref{lem-preren}, and the one-dimensional \thmref{th:main1}, imply that the operator $\cD$ is compact, and all of its
  non-zero eigenspaces lie inside the tangent bundle to $\iota(\bC^{D,E})$.

  The image of $\hat\cR_n$ in the one-dimensional subspace $\iota(\bC^{D,E})$ contains pairs for which the almost commutation condition holds, but which may not have a critical point of order $3$ at the origin. Clearly, almost commuting pairs $\bB^{D,E}$ have codimension $2$ in this space. By \thmref{th:main1},
the stable bundle of the operator $\cD$ restricted to the tangent bundle of $\iota(\bB^{D,E})$ has codimension $1$ -- together with the above, it gives the required bound.

\end{proof}

\section{Critical attractors of dissipative maps}\label{attractor}
As before, let $\cR^p(\zeta_*)=\zeta_*$.
Fix $\rho_*\equiv \rho(\zeta_*)\in(0,1)\setminus \QQ$. Set $T_a(x)\equiv x+a$, and
$$T_*\equiv (T_{\rho_*}|_{[-1,0]},T_{-1}|_{[0,\rho_*]}).$$
The main result of this section is the following theorem:
\begin{thm}
\label{th:attractor}
Let $\zeta_*=\cR^p(\zeta_*)$ be as above and let 
$$Z_*=(A_*,B_*)=\iota(\zeta_*)\in\bD^{\Omega,\Gamma}_{\eps,\delta}.$$
Suppose $Z=(A,B)\in W^s_{\text{loc}}(Z_*)\subset\bD^{\Omega,\Gamma}_{\eps,\delta},$ and suppose that maps $A$ and $B$ commute, that is $A\circ B=B\circ A$,
where defined (for instance, $A=F^{q_{n+1}},\; B=F^{q_n}$ could be iterates of the same map).

Then $Z$ has a minimal attractor $\Sigma$ in 
$\Omega\cup\Gamma$. The attractor $\Sigma$ is a Jordan arc, and the restriction $Z|_\Sigma$ is topologically but not smoothly conjugate to $T_*$.
\end{thm}
\begin{proof}

Below, we will denote $\Upsilon^1=\Omega,  \Upsilon^2=\Gamma$.
As in the previous section, $\hat\cR_n$ will denote the extension of $\cR^n$ to two-dimensional dissipative maps for some
$n=p m$  sufficiently large (how large will be fixed later).
For notational simplicity, we set
 $$\cRG=\hat\cR_n.$$

To differentiate the changes of coordiates corresponding to different pairs, given a pair $Z$, denote $\Lambda_Z$ the linear rescaling (\ref{lambdan})  in the definition of  $\hat\cR_n Z$, and $H_Z$ the non-linear change of coordinates (\ref{coordh}). By \propref{trivialproj},
$$\cRG Z=
\Lambda_Z^{-1}\circ H_Z \circ \hat p \cR^n Z \circ H_Z^{-1} \circ \Lambda_Z$$
with $\hat p\cR^n$ defined in (\ref{preren-crit1}).
Again by \propref{trivialproj}, for $l\in\NN$, we have:
$$\cRG^l Z =L_{\cRG^{l-1}Z}^{-1} \circ \ldots \circ L_Z^{-1} \circ \hat p \cR^{ ln} Z \circ L_{Z} \circ \ldots \circ L_{\cRG^{l-1} Z},$$
where
$$L_Z\equiv  H_Z^{-1} \circ \Lambda_Z.$$

Let $\bar s_n^l$ and $\bar t_n^l$ be defined by
$$(\hat p\cR^n)^l\zeta_*=(\zeta_*^{\bar s_n^l},\zeta_*^{\bar t_n^l}),$$
where $\hat p\cR^n\zeta_*$ is as in (\ref{preren-crit}).
For each  multi-index $\bar w=(a_0,b_0,a_1,b_1, \ldots, a_k,b_k) \prec \bar s_n^l$ or $\bar w=(a_1,b_1, \ldots, a_k,b_k) \prec \bar t_n^l$ we define a domain
\begin{equation} \label{Qiw}
  Q_{\bar w}^i=Z^{\bar w} \circ  L_Z \circ L_{\cRG  Z} \circ \ldots \circ L_{\cRG^{l-1} Z}(\Upsilon^i), \ i=1 \text{ for }\bar w\prec\bar s_n^l, \ i=2 \text{ for }\bar w\prec\bar t_n^l.
\end{equation}


\noindent
By analogy with a dynamical partition of a commuting pair, the collection 
$$\cQ_{ln}\equiv \{ Q_{\bar w}^i\}$$ will be refered to as the $l n$-th partition for the two-dimensional  pair $Z$.

Given $Z \in W^s_{\text{loc}}(Z_*)$, consider the following collection of  functions defined on $\Omega \cup \Gamma$:
$$\Psi_{\bar w}^Z= Z^{\bar w} \circ L_Z.$$

\noindent
Given a collection of index sets $\{\bar w^i\}$, $\bar w^i \prec \bar s_{n}$ or $\bar w^i \prec \bar t_{n}$, consider the following {\it renormalization microscope}
$$\Phi_{\bar w^0,\bar w^1, \bar w^2, \ldots, \bar w^{k-1},Z}^{k}= \Psi_{\bar w^0}^Z \circ \Psi_{\bar w^1}^{\cRG Z} \circ \ldots \circ  \Psi_{\bar w^{k-1}}^{\cRG^{(k-1)} Z},$$
which we will also denote $\Phi^k_{{\hat w_0^{k-1}},Z}$, $\hat w_0^{k-1}=\left\{\bar w^0,\bar w^1, \bar w^2, \ldots, \bar w^{k-1} \right\}$,  for brevity. 

\begin{lem} The renormalization microscope  maps a set $\Upsilon^i$ onto an element of partition $\cQ_{k n}$ for $Z$.
\end{lem}
\begin{proof}
The claim holds for $k=1$ by the definition $(\ref{Qiw})$ of the elements of the partition.

Assume that it $\Phi^k_{\hat w_0^{k},Z}(\Upsilon^i)$ is an element of partition $\cQ_{kn}$ for $Z$.

Consider  $\Phi^{k+1}_{\hat w_0^k,Z}(\Upsilon^i)$:
$$\Phi^{k+1}_{\hat w_0^{k},Z}(\Upsilon^i)=\Psi_{\bar w^0}^Z \circ \Psi_{\bar w^1}^{\cRG Z} \circ \ldots \circ  \Psi_{\bar w^k}^{\cRG^{k } Z}(\Upsilon^i).$$
By assumption, 
$$\Phi_{\hat w_1^{k},\cRG Z}^k(\Upsilon^i) \equiv \Psi_{\bar w^1}^{\cRG Z} \circ \ldots \circ  \Psi_{\bar w^k}^{\cRG^{k } Z}(\Upsilon^i)$$ 
is an element of the partition of level $k n$ for the pair $\cRG Z$, that is, by $(\ref{Qiw})$  
$$\Phi_{\hat w_1^{k},\cRG Z}^k(\Upsilon^i)=(\cRG Z)^{\bar v}  \circ  L_{\cRG Z} \circ L_{\cRG^2 Z} \circ \ldots \circ L_{\cRG^k Z} (\Upsilon^i),$$
for some admissible $\bar v=(\alpha_0,\beta_0, \alpha_1, \beta_1, \ldots, \alpha_m,\beta_m)$. Therefore, using the shorthand $$\cRG Z=(A_1,B_1),$$ we have:
\begin{eqnarray}
\nonumber \Phi^{k+1}_{\hat w_0^{k},Z}(\Upsilon^i)&=& \Psi_{\bar w^0}^Z \circ \Phi_{\hat w_1^{k},\cRG Z}^k(\Upsilon^i), \\
\nonumber &=& Z^{\bar w^0} \circ L_{Z} \circ (\cRG Z)^{\bar v} \circ  L_{\cRG Z}  \circ \ldots \circ L_{\cRG^k Z} (\Upsilon^i)  \\
\nonumber &=& Z^{\bar w^0} \circ L_{Z} \circ ( B_1^{\beta_m} \circ A_1^{\alpha_m} \circ \ldots \circ  B_1^{\beta_0} \circ A_1^{\alpha_0}) \circ  L_{\cRG Z}  \circ \ldots \circ L_{\cRG^k Z} (\Upsilon^i)  \\
\nonumber &=& Z^{\bar w^0} \circ L_{Z} \circ \Lambda_Z^{-1} \circ H_Z \circ \left( \left(Z^{\tilde t_n}\right)^{\beta_m} \circ \left(Z^{\tilde s_n}\right)^{\alpha_m} \circ \ldots \circ  \left(Z^{\tilde t_n}\right)^{\beta_0} \circ \left(Z^{\tilde s_n}\right)^{\alpha_0} \right) \circ \\
 \nonumber &\phantom{=}&  \phantom{Z^{\bar w^0}} \circ H_z^{-1} \circ \Lambda_Z \circ  L_{\cRG Z}  \circ \ldots \circ L_{\cRG^k Z} (\Upsilon^i)  \\
\nonumber &=& Z^{\bar w^0}  \circ \left(Z^{\tilde t_n}\right)^{\beta_m} \circ \left(Z^{\bar s_n'}\right)^{\alpha_m} \circ \ldots \circ \left(Z^{\bar t_n'}\right)^{\beta_0} \circ \left(Z^{\bar s_n'}\right)^{\alpha_0}  \circ L_Z   \circ \ldots \circ L_{\cRG^k Z} (\Upsilon^i)  \\
\nonumber &=& Z^{\bar u} \circ L_Z  \circ \ldots \circ L_{\cRG^k Z} (\Upsilon^i),
\end{eqnarray}
for some index $\bar u$. By  (\ref{Qiw}), the latter is an element of the partititon $\cQ_{(k+1) n}$.
\end{proof}


\noindent
Since ${\cRG^{l} Z}$ converges to $Z_*$ at a geometric rate, the function $\Psi_{\bar w}^{\cRG^l Z}$  converges to the function $\psi_{\bar w}^{\zeta_*}$, defined in Corollary~\ref{cor:partition},  at a geometric rate in $C^1$-metric. Therefore, by   Corollary~\ref{cor:partition}, there  exists a neighborhood $\cS$ in $W^s_{\text{loc}}(Z_*)$ of $Z_*$, and  sufficiently large $n=p m$ in the definition of $\cRG$ and $l$, such that 
$$\|D \Psi_{\bar w}^{\cRG^l Z} |_{\Upsilon^i} \|_\infty < {1 \over 2},$$
whenever $\cRG^l Z \in \cS$.

For every $Z \in W^s_{\text{loc}}(Z_*)$,  there exists $i_0\in\NN$ such that $\cRG^{i} Z \in \cS$ for $i\geq i_0$. Hence,  there exists $C=C(Z)$, such that
\begin{equation}
\label{miccontract}\|D \Phi^k_Z|_{\Upsilon^i} \|_\infty < {C \over 2^k},
\end{equation}
and thus the renormalization microscope is a uniform metric contraction.

We are now ready to finish the proof.

Select a distinct point $(x_{\bar w},y_{\bar w})$ in each of the sets $Q^i_{\bar w} \in  \cQ_{ln}$. Consider the $ln$-th dynamical partition $\cP_{l n}$ for the pair $T_*$ as defined in Section~\ref{sec:partition}.    Consider a piecewise-constant map  $\varphi_l$ sending the element of the partition with a multi-index $\bar w$ to $(x_{\bar w},y_{\bar w})$. By (\ref{miccontract}), the diameters of the sets $Q^i_{\bar w}$ decrease at a geometric rate. Thus,  the maps $\varphi_l$  converge uniformly to a continuous map $\varphi$ of the interval $[-1,\rho_*]$ which is a homeomorphism onto the image. Set 
$$\varphi([-1,\rho_*])\equiv \Sigma.$$
By construction,
$$\varphi\circ T_*=Z \circ \varphi,$$
and the curve $\Sigma$ is the attractor for the pair $Z$. 
Clearly, the conjugacy $\varphi$ cannot be $C^1$-smooth, since the limiting pair $\zeta_*$ has a critical point at the origin.

\end{proof}

\appendix

\section{Proof of \propref{prop:2Dprojection}}
\label{Apx1}
Let us write
$$\Pi (\tilde A,\tilde B)(x,y)=\left( \left( \tilde \eta_1(x) +a x^4+b x^6 + \tilde \tau_1(x,y) \atop   \tilde \eta_2(x) +a x^4+b x^6+ \tilde \tau_2(x,y)  \right), \left( \tilde \xi_1(x)+c + d x + e x^2 +  \tilde \pi_1(x,y) \atop   \tilde \xi_2(x) +c+d x +e x^2 + \tilde \pi_2(x,y)  \right)   \right).$$

The conditions (\ref{commutation})-(\ref{normalization}) translate into the 
the following system of $5$ equations $\bF(a,b,d,e,c)=0$:
\begin{eqnarray}
\nonumber a&+&b-d \tilde \eta_1(0)-e \tilde \eta_1(0)^2 -c-\left(\tilde \eta _1(\tilde \xi_1(0))-\tilde \eta_1(\tilde \xi_1(0)+c) \right)-\\
\nonumber &-& \left\{{ \tilde \tau_1(\tilde \xi_1(0), \tilde \xi_2(0))-\tilde \tau_1(\tilde \xi_1(0)+c, \tilde \xi_2(0)+c)} \right\}\\
\nonumber &=& \pi_1( \tilde B \circ \tilde A(0,0)- \tilde A \circ \tilde B(0,0))
\end{eqnarray}
\begin{eqnarray}
\nonumber (\tilde \xi_1'(0) &+&d) (4 a +6 b)+\tilde \eta_1'(\tilde \xi_1(0)+c) (\tilde \xi_1'(0)+d)-\tilde \eta_1'(\tilde \xi_1(0)) \tilde \xi_1'(0)+\\
\nonumber &+&\tilde \xi_1'(\tilde \eta_1(0)) \tilde \eta_1'(0)-(\tilde \xi_1'(\tilde \eta_1(0))+d+2 e \tilde \eta_1(0)) \tilde \eta_1'(0)-\\
\nonumber &+&\left\{{\nabla \tilde \tau_1(\tilde \xi_1(0)+c,\tilde \xi_2(0)+c) \cdot (\tilde \xi_1'(0)+d, \tilde \xi_2'(0)+d)-\nabla \tilde \tau_1(\tilde \xi_1(0),\tilde \xi_2(0)) \cdot (\tilde \xi_1'(0), \tilde \xi_2'(0))}\right\}\\
\nonumber & =& \pi_1( \tilde B \circ \tilde A(x,0)- \tilde A \circ \tilde B(x,0))'\arrowvert_{x=0} 
\end{eqnarray}
\begin{eqnarray}
\nonumber (\tilde \xi_1'(0)&+&d)^2 (12 a +30 b)+(\tilde \xi_1''(0)+2 e) (4 a +6 b)+\\
\nonumber &+&\tilde \eta_1''(\tilde \xi_1(0)+c) (\tilde \xi_1'(0)+d)^2+\tilde \eta_k'(\tilde \xi_1(0)+c) (\tilde \xi_1''(0)+2 e)-\\
\nonumber &-&(\tilde \xi_1''(\tilde \eta_1(0)) +2 e)\tilde \eta_1'(0)^2-(\tilde \xi_1'(\tilde \eta_1(0))+d +2 e \tilde \eta_1(0))  \tilde \eta_1''(0)-\\
\nonumber &-&\tilde \eta_1''(\tilde \xi_1(0)) \tilde \xi_1'(0)^2-\tilde \eta_1'(\tilde \xi_1(0)) \tilde \xi_1''(0)+\tilde \xi_1''(\tilde \eta_1(0))\tilde \eta_1'(0)^2+\tilde \xi_1'(\tilde \eta_1(0))  \tilde \eta_1''(0)\\
\nonumber&+&\left\{{\sum_{i,j=1,2}\partial_{i,j}\tilde \tau_1(\tilde \xi_1(0)+c,\tilde \xi_2(0)+c) \tilde (\xi_i'(0)+d) (\tilde \xi_j'(0)+d)+}\right.\\
\nonumber &+&{ \nabla \tilde \tau_1(\tilde \xi_1(0)+c,\tilde \xi_2(0)+c) \cdot (\tilde \xi_1''(0)+2 e, \tilde \xi_2''(0) +2 e)-}\\
\nonumber&-&\left.{\sum_{i,j=1,2}\partial_{i,j}\tilde \tau_1(\tilde \xi_1(0),\tilde \xi_2(0)) \tilde (\xi_i'(0)) (\tilde \xi_j'(0))-\nabla \tilde \tau_1(\tilde \xi_1(0),\tilde \xi_2(0)) \cdot (\tilde \xi_1''(0), \tilde \xi_2''(0)) }\right\}\\
\nonumber & =& \pi_1( \tilde B \circ \tilde A(x,0)- \tilde A \circ \tilde B(x,0))''\arrowvert_{x=0} 
\end{eqnarray}
\begin{eqnarray}
\nonumber (\tilde \xi_1'(0)&+&d) (24 a +120 b) + 3 (\tilde \xi_1'(0)+d) (\tilde \xi_1''(0)+2 e) (12 a +30 b)+ \tilde \xi_1'''(0) (4 a +6 b) +\\
\nonumber &+&\tilde \eta_1'''(\tilde \xi_1(0)+c) (\tilde \xi_1'(0)+d)-\tilde \eta_1'''(\tilde \xi_1(0)) \tilde \xi_1'(0)+\\
\nonumber &+&3 \tilde \eta_1''(\tilde \xi_1(0)+c) (\tilde \xi_1'(0)+d) (\tilde \xi_1''(0)+2 e)-3 \tilde \eta_1''(\tilde \xi_1(0)) \tilde \xi_1'(0) \tilde \xi_1''(0)+\\
\nonumber &+&\tilde \eta_1'(\tilde \xi_1(0)+c) \tilde \xi_1'''(0)-\tilde \eta_1'(\tilde \xi_1(0)) \tilde \xi_1'''(0)-\\
\nonumber &-&3 (\tilde \xi_1''(\tilde \eta_1(0))+2 e) \tilde \eta_1'(0) \tilde \eta_1''(0)+3 \tilde \xi_1''(\tilde \eta_1(0)) \tilde \eta_1'(0) \tilde \eta_1''(0)-\\
\nonumber &-&(\tilde \xi_1'(\tilde \eta_1(0))+d +2 e \tilde \eta_1(0) )  \tilde \eta_1'''(0)+\tilde \xi_1'(\tilde \eta_1(0)) \tilde \eta_1'''(0)-\\
\nonumber&+&\left\{{\sum_{i,j,k=1,2}\partial_{i,j,k}\tilde \tau_1(\tilde \xi_1(0)+c,\tilde \xi_2(0)+c) \tilde (\xi_i'(0)+d) (\tilde \xi_j'(0)+d)( \tilde \xi_k'(0)+d)+} \right.\\
\nonumber&{+}&{3 \sum_{i,j=1,2}\partial_{i,j}\tilde \tau_1(\tilde \xi_1(0)+c,\tilde \xi_2(0)+c) (\tilde \xi_i''(0)+2 e) (\tilde \xi_j'(0)+d)}\\
\nonumber &{+}&{ \nabla \tilde \tau_1(\tilde \xi_1(0)+c,\tilde \xi_2(0)+c) \cdot (\tilde \xi_1'''(0), \tilde \xi_2'''(0))-\sum_{i,j,k=1,2}\partial_{i,j,k}\tilde \tau_1(\tilde \xi_1(0),\tilde \xi_2(0)) \tilde \xi_i'(0)\tilde \xi_j'(0)\tilde \xi_k'(0)-}\\
\nonumber&{-}&\left.{3 \sum_{i,j=1,2}\partial_{i,j}\tilde \tau_1(\tilde \xi_1(0),\tilde \xi_2(0)) \tilde \xi_i''(0) \tilde \xi_j'(0)- \nabla \tilde \tau_1(\tilde \xi_1(0),\tilde \xi_2(0)) \cdot (\tilde \xi_1'''(0), \tilde \xi_2'''(0))} \right\}\\
\nonumber & =& \pi_1( \tilde B \circ \tilde A(x,0)- \tilde A \circ \tilde B(x,0))''\arrowvert_{x=0} 
\end{eqnarray}
\begin{eqnarray}
\nonumber c&=&1-\tilde \xi_1(0).
\end{eqnarray}
The functions in the parenthesis above have the uniform norm $O(\eps^2) \cdot \max\{c,d,e\}$.

Notice, that when the commutator $\pi_1 \left( A \circ B - B \circ  A \right)(x,0)=o(|x|^3)$ and $B$ is normalized appropriately, $B(0,0)=(1,1)$, this system of equations is solved by $a=b=d=e=c=0$. Furthermore, denote $\mathbf{p}=(a,b,d,e,c)$, then the derivative $D_{\mathbf p} \bF(\mathbf 0)$ is given by
$$
\left[
\begin{array}{c c c c c}
\vspace{4mm}
1 & 1 & -\eta_1(0)& -\eta_1(0)^2& a_{1,5} \\
\vspace{4mm}
4 \vareps_1 & 6 \vareps_1& { \tilde \eta_1'(\tilde \xi_1(0))- \atop \nu_1+\delta_1} & -2 \tilde \eta_1(0)  \nu_1 & a_{2,5} \\
\vspace{4mm}
12 \vareps_1^2 + 4 \alpha_1 & 30 \vareps_1^2 +6 \alpha_1 & {2 \vareps_1 \tilde \eta_1''(\tilde \xi_1(0))-\beta_1 \atop +\delta_2} &{2 \tilde \eta_1'(\tilde \xi_1(0)) - \atop 2  \nu_1^2-2 \tilde \eta_1(0)  \beta_1 +\delta_3}  & a_{3,5}\\ 
\vspace{4mm}
{4 \tilde \xi_1'''(0)+ \atop 24 \vareps_1 + 36 \vareps_1 \alpha_1}& {6 \tilde \xi_1'''(0)+ \atop 120 \vareps_1 + 90 \vareps_1 \alpha_1} & {\tilde \eta_1'''(\tilde \xi_1(0)) -\tilde \eta_1'''(0)  \atop + 3 \tilde \eta_1''(\tilde \xi_1(0)) \alpha_1+\delta_4} & {-2 \tilde \eta_1(0) \tilde \eta_1'''(0)- \atop -6 \tilde \eta_1''(\tilde \xi_1(0)) \vareps_1-6 \beta_1 \nu_1 + \delta_5} & a_{4,5}\\ 
\vspace{4mm}
0 & 0 & 0 & 0 & 1
\end{array}
\right],
$$
where $a_{i,5}$ denote certain bounded numbers whose values are irrelevant for the computation of the determinant, $\vareps_i=\tilde \xi_i'(0)$, $\nu_i=\eta'_i(0)$, $\alpha_i=\tilde \xi''_i(0)$, $\beta_i=\tilde \eta_i''(0)$, $i=1,2$, while $\delta_i$ are some number whose size is $O(\eps^2$).

The determinant of the above matrix is $\max\{\vareps_1,\nu_1,\alpha_1,\beta_1,\eps  \}$-close to $4 (\tilde \eta_1'(\tilde \xi_1(0)))^2 \tilde \xi_1'''(0)$ and is nonzero for $(A,B) \in \iota \bD^{\Omega,\Gamma}_{\eps,\delta}$ if $\eps$ is sufficiently small. The claim follows by an application of the Regular Value Theorem.


\end{document}